\documentclass[11pt]{amsart}
\usepackage{ifthen,latexsym,amssymb,amsmath,bbm,fixmath}
\usepackage{verbatim}
\usepackage{scalefnt}
\usepackage{multirow}
\usepackage{float}
\restylefloat{figure}
\usepackage[margin=1.3in]{geometry}
\usepackage{fullpage}

\usepackage{pstricks,tikz}

\usepackage{fancyhdr}
\usepackage{hyperref}
\usepackage{mathrsfs}

\newcommand{\V}[1]{\mathbold{#1}}

\begin{document}

\def\COMMENT#1{}

\newtheorem{theorem}{Theorem}
\newtheorem{lemma}[theorem]{Lemma}
\newtheorem{proposition}[theorem]{Proposition}
\newtheorem{corollary}[theorem]{Corollary}
\newtheorem{conjecture}[theorem]{Conjecture}
\newtheorem{claim}{Claim}[theorem]
\newtheorem{definition}[theorem]{Definition}
\newtheorem{question}[theorem]{Question}

\newenvironment{claimproof}[1]{\par\noindent\underline{Proof:}\space#1}{\leavevmode\unskip\penalty9999 \hbox{}\nobreak\hfill\quad\hbox{$\blacksquare$}}


\def\eps{{\varepsilon}}
\newcommand{\cP}{\mathcal{P}}
\newcommand{\cT}{\mathcal{T}}
\newcommand{\cL}{\mathcal{L}}
\newcommand{\ex}{\mathbb{E}}
\newcommand{\eul}{e}
\newcommand{\pr}{\mathbb{P}}

\title[The bandwidth theorem for locally dense graphs]{The bandwidth theorem for locally dense graphs}
\author{Katherine Staden and Andrew Treglown}
\thanks{K.S.\ was supported by ERC grant~306493. A.T.\ was supported by EPSRC grant EP/M016641/1.}

\begin{abstract}
The \emph{Bandwidth theorem} of B\"ottcher, Schacht and Taraz~[{\it Proof of the bandwidth conjecture of Bollob\'as and Koml\'os,  Mathematische Annalen, 2009}]  gives a condition on the minimum degree of an $n$-vertex graph $G$
 that ensures $G$ contains every $r$-chromatic graph $H$ on $n$ vertices of bounded degree
and of bandwidth $o(n)$, thereby proving a conjecture of Bollob\'as and Koml\'os~[{\it The Blow-up Lemma, Combinatorics, Probability and Computing, 1999}].
In this paper we prove a version of the Bandwidth theorem for \emph{locally dense} graphs.
Indeed, we  prove that every locally dense $n$-vertex graph $G$ with  $\delta (G) > (1/2+o(1))n$
contains as a subgraph any given (spanning) $H$ with bounded maximum degree and sublinear bandwidth.
\end{abstract}

\date{\today}
\maketitle

\section{Introduction and results}
One of the fundamental topics in extremal graph theory is the study of minimum degree conditions that force a graph to contain a given spanning substructure. Perhaps the best known result in the area is Dirac's theorem~\cite{dirac}, which
states that any graph $G$ on $n\geq 3$ vertices with minimum degree $\delta (G) \geq n/2$ contains a Hamilton cycle. 
The P\'osa--Seymour conjecture (see~\cite{posa} and~\cite{seymour}) states that any graph $G$ on $n$ vertices with $\delta (G) \geq
rn/(r+1)$ contains the $r$th power of a Hamilton cycle. (The $r$th power of a Hamilton cycle $C$ is obtained from $C$ by adding an edge
between every pair of vertices of distance at most $r$ on $C$.)  Koml\'os, S\'ark\"ozy and Szemer\'edi~\cite{kss} proved this conjecture
 for sufficiently large graphs.

A decade ago, B\"ottcher, Schacht and Taraz~\cite{bot} proved a very general minimum degree result, the so-called \emph{Bandwidth  theorem}. 
A graph $H$ on $n$ vertices is said to have \emph{bandwidth at most $b$}, if there exists a 
labelling of the vertices of $H$ by the numbers $1, \dots ,n$ such that for every edge $ij\in E(H)$ we have $|i-j|\leq b$.
Clearly every graph $H$ has bandwidth at most $|H|-1$. 
Further, a Hamilton cycle has bandwidth $2$, and in general the $r$th power of a Hamilton cycle has bandwidth at most $2r$.
B\"ottcher, Preussmann,  Taraz and  W\"urfl~\cite{BPTW10} proved that every planar graph  $H$ on $n$ vertices with bounded maximum degree has bandwidth at most $O(n/\log n)$.
The Bandwidth theorem gives a condition on the minimum  degree of a graph $G$ on 
$n$ vertices
that ensures $G$ contains \emph{every} $r$-chromatic graph on $n$ vertices of bounded degree and of bandwidth $o(n)$.
\begin{theorem}[The Bandwidth  theorem, B\"ottcher, Schacht and Taraz~\cite{bot}]\label{bst}
Given any $r,\Delta \in \mathbb N$ and any $\gamma >0$, there exist constants $\beta >0$ and $n_0 \in \mathbb N$
such that the following holds. Suppose that $H$ is an $r$-chromatic graph on $n \geq n_0$ vertices with $\Delta (H) \leq \Delta$ 
and bandwidth at most $\beta n$. If $G$ is a graph on $n$ vertices with
$$\delta (G) \geq \left( \frac{r-1}{r}+\gamma \right)n,$$
then $G$ contains a copy of $H$. 
\end{theorem}
We remark that Theorem~\ref{bst} had been conjectured by  Bollob\'as and Koml\'os~\cite{komlos}.
Since the Bandwidth  theorem was proven, a number of variants of the result have been obtained including for arrangeable graphs~\cite{arrange}, degenerate graphs~\cite{lee} and in the setting of random and pseudorandom graphs~\cite{abet, bkt, hls}, as well as
for robustly expanding graphs~\cite{knox}.
Very recently a Bandwidth  theorem for approximate decompositions was proven by Condon, Kim, K\"uhn and Osthus~\cite{ckko}, whilst Glock and Joos~\cite{gj} proved a $\mu n$-bounded edge colouring extension of Theorem~\ref{bst}.
A general embedding result of B\"ottcher, Montgomery, Parczyk and Person~\cite{bmpp} also implies a bandwith theorem in the setting of randomly perturbed graphs.

For many graphs $H$, the minimum degree condition in Theorem~\ref{bst} is best-possible up to the term $\gamma n$. 
For example, suppose that $H$ is a \emph{$K_r$-factor} (that is, we seek a collection of vertex-disjoint copies of $K_r$ in $G$ that together cover
all the vertices in $G$). So $\chi(H)=r$, $\Delta(H) =r-1$ and $H$ has bandwidth $r-1$. 
Suppose that $G$ is obtained from two disjoint vertex classes $A$ and $B$ of sizes $n/r+1$ and $(r-1)n/r-1$ respectively so that $G$ contains all edges other than those with both endpoints in $A$.
 Then it is easy to see
that $G$ does not contain a $K_r$-factor, however, $\delta (G) = \left( \frac{r-1}{r} \right)n-1$. In fact, note that the famous Hajnal--Szemer\'edi theorem~\cite{hs} asserts that an $n$-vertex graph $G$ contains a $K_r$-factor provided
$r|n$ and $\delta (G)\geq \left( \frac{r-1}{r} \right)n$. Thus, this extremal example is sharp.
(Note though for many $r$-partite graphs $F$, a significantly lower minimum degree condition than that in Theorem~\ref{bst} ensures an $F$-factor, see~\cite{kuhn2}.)

As for many other problems in the area, this extremal example has the characteristic that it contains a \emph{large} independent set.
There has thus been significant interest in seeking variants of classical results in extremal graph theory, where one now forbids the host graph from containing a large independent set.
Indeed, nearly 50 years ago, Erd\H{o}s and S\'os~\cite{sos} initiated the study of the Tur\'an problem under the additional assumption of small independence number. That is, they considered the
number of edges in an $n$-vertex $K_r$-free graph with independence number $o(n)$. This topic is now known as \emph{Ramsey-Tur\'an theory} and has been extensively studied by numerous authors (see e.g. \cite{lenz,ehss,mum, sim}).
More recently, there has been interest in similar questions but where now one seeks a $K_r$-factor in an $n$-vertex graph with independence number $o(n)$ and large minimum degree (see~\cite{triangle, bms,  han}).

A stronger notion of a graph not containing a large independent set, is that of being \emph{locally dense}. More precisely, given $\rho,d >0$, we say that an $n$-vertex graph $G$ is \emph{$(\rho,d)$-dense} 
if every $X \subseteq V(G)$  satisfies $e(G[X]) \geq d\binom{|X|}{2}-\rho n^2$. 
Note that the property of being locally dense is weaker than being dense and (pseudo)random and stronger than having sublinear independence number.
Locally dense graphs have been considered in a number of previous papers. For example, there have been several papers on a question of Erd\H{o}s, Faudree, Rousseau, and Schelp~\cite{efrs}; there they considered a  variant of the notion of $(\rho,d)$-dense, and asked for the values of $\rho$ and $d$ that
guarantee a $(\rho,d)$-dense graph contains a triangle.
One can view the notion of locally dense as a parameter that ensures a graph is in some sense `random-like'. Therefore, there has been interest in determining the number of (homomorphic) copies of a fixed graph $H$ in a $(\rho,d)$-dense graph $G$, and in particular whether this count is close to the value obtained if $G$ were a random graph; the study of this topic (for graphs {and} hypergraphs) was initiated by
Kohayakawa,  Nagle,  R\"odl and Schacht~\cite{knrs}.

The aim of this paper is to prove the following locally dense version of the Bandwidth  theorem.
\begin{theorem}\label{main}
For all $\Delta \in \mathbb{N}$ and $d,\eta>0$, there exist constants $\rho,\beta,n_0>0$ such that for every $n \geq n_0$, the following holds.
Let $H$ be an $n$-vertex graph with $\Delta(H) \leq \Delta$ and bandwidth at most $\beta n$.
Then any $(\rho,d)$-dense graph $G$ on $n$ vertices with $\delta(G)\geq (1/2+\eta)n$ contains a copy of $H$.
\end{theorem}
In the case when $H$ corresponds to a $K_r$-factor, Theorem~\ref{main} had been proven by Reiher and Schacht (see~\cite{bms}).
Note that in the case when $H$ is connected, the minimum degree in Theorem~\ref{main} is best-possible, up to the $\eta n$ term. Indeed, if $G$ consists of two vertex-disjoint cliques each of size $n/2$ then
$G$ trivially does not contain $H$ though $G$ is locally dense and $\delta (G)=n/2-1$.

A striking feature of Theorem~\ref{main} is that, unlike Theorem~\ref{bst}, the minimum degree condition does not depend on the chromatic number of $H$.
In particular, when $\chi (H)=2$, the minimum degree condition in Theorem~\ref{main} is the same as that in Theorem~\ref{bst}. Thus, in the case of bipartite $H$, there is no benefit in adding the condition that $G$ is locally dense. However, when $\chi (H)>2$, the minimum degree in Theorem~\ref{main} is substantially reduced compared to the Bandwidth  theorem. 

It would also be extremely interesting to prove a version of Theorem~\ref{main} for graphs of sublinear independence number.
Note though that examples in~\cite{bms} show that the statement of Theorem~\ref{main} is far from true if we require that $G$ has sublinear independence number instead of the locally dense condition.
 Indeed, the minimum degree necessary for the existence of a $K_r$-factor in such a graph is at least $(\frac{r-2}{r}+o(1))n$ for every $r \geq 4$.
So these two problems are genuinely different.



The proof of Theorem~\ref{main} draws on ideas from~\cite{3partite, bot}, and our approach makes use of the Regularity--Blow-up method.
We also employ several new ideas (particularly with regard to dealing with so-called exceptional vertices). In the next section we give an overview of the proof of Theorem~\ref{main}.
In Section~\ref{sec3} we introduce some notation, as well as several fundamental properties of locally dense graphs. The Regularity and Blow-up lemmas
are presented in Section~\ref{secrl}. A key step in the proof of Theorem~\ref{main} is to show that the hypothesis of this theorem ensures $G$ contains the $r$th power of a Hamilton cycle; we prove this in Section~\ref{cyclesec}.
The proof of Theorem~\ref{main} then breaks into two main parts: the proof of two so-called Lemmas for $H$ (presented in Section~\ref{sec6}) and the Lemma for $G$ (presented in Section~\ref{sec7}). In Section~\ref{sec8} we combine all these results
to prove Theorem~\ref{main}. We give some concluding remarks in Section~\ref{sec9}.

{\bf Additional note.} Since this paper was first submitted,
 Ebsen, Maesaka,  Reiher, Schacht and Sch\"ulke~\cite{newpaper} have built on our work to generalise Theorem~\ref{main}. Indeed, they replace the minimum degree condition on $G$ with an inseparability condition.

\section{Overview of the proof of Theorem~\ref{main}}\label{sketch}
The overall strategy follows in the same spirit as the proof of the Bandwidth theorem in~\cite{bot}, though the precise details of the proofs of the key steps in the argument turn out to be quite different.
In particular, the setting of locally dense graphs both smooths over some aspects of the proof, as well
as introducing additional difficulties. 
Often in problems involving embedding a spanning structure, the most challenging aspect of the proof is dealing with so-called exceptional vertices (i.e. either trying to cover the remaining last few vertices in the host graph or those few vertices that do not fit in some general structure in the host graph). In this paper, we take a novel approach to dealing with such vertices.
Below we outline the key steps in our proof and highlight
some of the main novelties in our approach.

\smallskip

\noindent
{\bf Obtaining structure in $G$.} 
Suppose that $H$ and $G$ are as in the statement of the theorem where $\chi (H)=r$.
The first step in the proof is to apply the Regularity lemma (Lemma~\ref{reg}) to $G$ to obtain the reduced graph $R$ of $G$. The reduced graph $R$ is locally dense (with somewhat different parameters compared to $G$) and 
`inherits' the minimum degree of $G$ (i.e. $\delta (R) > (1/2+o(1))|R|$). These properties ensure that $R$ contains an almost spanning subgraph $Z^{2r}_{\ell}$ which has the following  properties:
\begin{itemize}
\item $Z^{2r}_{\ell}$ covers all but at most $2r$ of the vertices in $R$;
\item $Z^{2r}_{\ell}$ consists of $\ell$ vertex-disjoint copies $K^1,\dots, K^{\ell}$ of $K_{2r}$ (in particular $|Z^{2r}_{\ell}|=2r\ell$);
\item For each $1\leq i \leq \ell$, there are all possible edges between $K^{i}$ and $K^{i+1}$ except that we miss a perfect matching between the two. (Note here $K^{\ell+1}:=K^{1}$.)
\end{itemize}
The existence of $Z^{2r}_{\ell}$ in $R$ can be guaranteed by finding a sufficiently large power of a Hamilton cycle in $R$. This is achieved in Section~\ref{cyclesec} (see Theorem~\ref{rcycle}).
Using this, one can easily deduce that $G$ contains an \emph{almost} spanning structure $\mathcal C$ that looks like the `blow-up' of $Z^{2r}_{\ell}$. More precisely, if $V(Z^{2r}_{\ell})=\{1,\dots, 2r\ell\}$ and
$V_1,\dots, V_{2r\ell}$ are the corresponding clusters in $G$, then
\begin{itemize}
\item[(i)] $V(\mathcal C)= V_1 \cup \dots \cup V_{2r\ell}$;
\item[(ii)] $\mathcal C [V_i,V_j]$ is $\eps$-regular whenever $ij \in E(Z^{2r}_{\ell})$;
\item[(iii)] If $jk$ is an edge in one of the cliques $K^i$ then $\mathcal C [V_j,V_k]$ is superregular.
\end{itemize}
We refer to $\mathcal C$ as a \emph{cycle structure}.

Suppose that in fact $\mathcal C$ is a spanning subgraph of $G$. In this case, ideally, one would now like to take the following approach. Let $x_1,\dots, x_n$ denote the bandwidth ordering of $H$. Partition $V(H)$ into $\ell$ classes $C_1,\dots, C_{\ell}$ so that 
\begin{itemize}
\item $c_i:=|C_i|=|\cup  _{j \in V(K^i)} V_j|$ for all $1\leq i \leq \ell$;
\item $C_1$ contains the vertices $x_1,\dots, x_{c_1}$; $C_2$ contains the vertices $x_{c_1+1},\dots, x_{c_1+c_2}$ and so forth.
\end{itemize}
Then embed the vertices from $C_1$ into the clusters in $G$ corresponding to the clique $K^1$, embed the vertices from $C_2$ into the clusters in $G$ corresponding to the clique $K^2$, and so on.

At first sight this seems like a plausible strategy: since the partition of $V(H)$ respected the bandwidth ordering of $H$ (and as $H$ has small bandwidth), most edges in $H$ lie in the induced subgraphs $H[C_i]$; all remaining edges lie in the bipartite graphs $H[C_i,C_{i+1}]$. Suppose one could map each $C_i$ onto the clusters corresponding to $K^i$, so that each such cluster $V_j$ receives precisely $|V_j|$ vertices from $C_i$,
and crucially, all edges $xy$ in $H[C_i]$ are such that $x$ and $y$ are mapped to different clusters in $K^i$.
That is, suppose
 we have a graph homomorphism $\phi _i$ between $H[C_i]$ and $K^i$ that maps precisely $|V_j|$ vertices to each $V_j$.
Further, suppose the $\phi _i$ together combine to give a graph homomorphism $f$ from $H$ to $Z^{2r}_{\ell}$ (so the edges in each $H[C_i,C_{i+1}]$ are mapped to edges in $R[V(K^i),V(K^{i+1})]$).
Set $G_i:=G[\cup  _{j \in V(K^i)} V_j]$. 
Then (iii) above ensures we could  apply the Blow-up lemma to each graph $G_i$ so as to embed $H[C_i]$ into $G_i$. Further, (ii) ensures that we can achieve this embedding so all edges in the graphs $H[C_i,C_{i+1}]$ are also present. That is, we would obtain an embedding of $H$ into $G$.

This naive approach fails though as there is no guarantee one can map each $C_i$ onto the clusters corresponding to $K^i$ so that each such cluster $V_j$ receives precisely $|V_j|$ vertices from $C_i$.
Furthermore, in the above approach we assumed that $\mathcal C$ is a spanning subgraph of $G$; in reality we have a small exceptional set $V_0$ of vertices in $G$ uncovered by $\mathcal C$.

\smallskip

\noindent
{\bf The Basic Lemma for $H$ and the Lemma for $G$.} 
Instead of the above, we prove the so-called Basic~Lemma~for~$H$ (Lemma~\ref{lemmaforH1}). Here we show that one can find a graph homomorphism $f$ from $H$ into $Z^{2r}_{\ell}$ so that for every cluster $V_i$ of $R$, \emph{approximately} $|V_i|$ vertices are mapped to it.
This therefore `almost' gives us the desired graph homomorphism $f$ from $H$ into $Z^{2r}_{\ell}$. 
In the proof of Lemma~\ref{lemmaforH1} we rely on the fact that the $K^i$ in $Z^{2r}_{\ell}$ are copies of $K_{2r}$; note that in the analogous structure in the proof of the  Bandwidth theorem~\cite{bot}, the $K^i$ are copies of $K_r$.
To see why our  condition is helpful for us, note that whilst in general an $r$-partite graph $H'$ does not have an `almost balanced' graph homomorphism into $K_r$ (since $H'$ may have colour classes of wildly different sizes), for  $r$-partite graphs $H'$ of bounded degree and sublinear bandwidth
one can 
always find an almost balanced graph homomorphism from $H'$ into $K_{2r}$.

Next, in the Lemma~for~$G$ (Lemma~\ref{lemmaforG}) we prove that, if one does not have an  exceptional set $V_0$, then we can move vertices around  the cycle structure $\mathcal C$ in such a way to ensure  that now each cluster $V_i$ in $\mathcal C$ has size precisely corresponding 
to the number of vertices mapped to $V_i$ by $f$. This is at the expense of weakening condition (iii): after applying Lemma~\ref{lemmaforG} we only have that each clique $K^i$ splits into two cliques $K^i _1$ and $K^i _2$ of size $r$ such that 
if $jk$ is an edge in one of the cliques $K_1^i$ or $K_2 ^i$ then $\mathcal C [V_j,V_k]$ is superregular.
However, this is still good enough to apply the Blow-up lemma to find our desired embedding of $H$ into $G$.

\smallskip

\noindent
{\bf Special Lemma for $H$.} So far we have been assuming that there is no exceptional set $V_0$; further, in the the proof of the  Bandwidth theorem~\cite{bot}, B\"ottcher, Schacht and Taraz were able to utilise the large minimum degree to incorporate exceptional vertices into (their analogue of the cycle structure) $\mathcal C$. We have significantly smaller minimum degree, so are unable to do this in our setting.

Instead, given the bandwidth ordering $x_1,\dots, x_n$ of $H$, we reserve a short initial segment $x_1,\dots, x_t$; and let $H'$ denote the subgraph of $H$ induced by $x_1,\dots, x_t$. 
Here $t$ will be significantly bigger than $\beta n$ (recall $H$ has bandwidth at most $\beta n$), but still $H'$ will only be a small fraction of $H$.
Via the Special Lemma for $H$ (Lemma~\ref{lemmaforH}) we will embed $H'$ into $G$ in such a way that crucially all of $V_0$ is covered by $H'$, and equally importantly, we do not cover more than a small proportion of each cluster $V_i$ in $\mathcal C$.

In the proof of Lemma~\ref{lemmaforH},
since $V_0$ may only contain very few (or even no edges) we must assign an independent set $I$ in $H'$ of size $|V_0|$ to be embedded onto $V_0$.
We then must connect up $I$ through the rest of $G$ to obtain a copy of $H'$. In particular, since $H'$ is much smaller than $H$, the distance between two vertices $x, y \in I$ in $H'$ may also be small. So it is crucial that $G$ is `highly connected'.
The Connecting lemma (Lemma~\ref{connect}) ensures this is the case. (Lemma~\ref{connect} is also applied in the proof of Theorem~\ref{rcycle}.)

Care is also needed to ensure that Lemma~\ref{lemmaforH} is compatible with the Basic~Lemma~for~$H$ (Lemma~\ref{lemmaforH1}). That is, we use Lemma~\ref{lemmaforH} to embed $H'$ in $G$ and Lemma~\ref{lemmaforH1} to embed $H\setminus H'$ in $G$. Thus, we need to ensure the copies of $H'$ and $H\setminus H'$ can be positioned in $G$ in such a way that they `glue' together  to form a copy of $H$.


\smallskip

Note that the reader should view the above overview as an idealisation of the proof.
Indeed, when we prove Theorem~\ref{main} in Section~\ref{sec8}, some of the details will be a little different. For example, for technical reasons it is in fact important that we find a spanning copy of $Z^{r^*}_{\ell}$ in $R$ for some $r^* \gg r$ rather than $Z^{2r}_{\ell}$.

\section{Preliminaries}\label{sec3}

\subsection{Notation}
Given a set $X$ and $k \leq |X|$, write $\binom{X}{k}$ for the set of $k$-element subsets of $X$. Given $r \in \mathbb{N}$, write $[[2r]]^2 := [r]^2 \cup ([2r]\setminus[r])^2$.
Given a function $f : X \rightarrow Y$ and $A \subseteq X$, we write $f|_A$ for the restriction of $f$ to $A$ and $f(A) := \lbrace f(a) : a \in A\rbrace$.

Given a graph $G$, we write $V(G)$ and $E(G)$ for the  vertex and edge sets respectively, and $|G| := |V(G)|$ and $e(G) := |E(G)|$.
The \emph{degree} of a vertex $x \in V(G)$ is denoted by $d_G(x)$ and its neighbourhood  by $N_G(x)$. The \emph{degree} of a subset $X \subseteq V(G)$ is $d_G(X) := |\bigcap_{x \in X}N_G(x)|$.
A subgraph $H \subseteq G$ is \emph{$s$-extendable} if $d_G(V(H)) \geq s$.
Given  vertices $x_1,\dots, x_k$ we write $N_G(x_1,\dots,x_k):=\bigcap_{1\leq i\leq k}N_G(x_i)$.
If $A \subseteq V(G)$ we write $N_G(x,A):=N_G(x) \cap A$ and $d_G(x,A) := |N_G(x) \cap A|$.
We say that $A$ is \emph{$k$-independent} if every vertex in $A$ is at distance at least $k+1$ in $G$, i.e.~the shortest path in $G$ between any pair of elements in $A$ has length at least $k+1$.
Given $X,Y \subseteq V(G)$ (not necessarily disjoint), define $e_G(X,Y)$ to be the number of edges $xy \in E(G)$ with $x \in X$ and $y \in Y$.
If $X$ and $Y$ are disjoint then $G[X,Y]$ is the bipartite graph with vertex classes $X$ and $Y$ whose edge set consists of all those edges in $G$ with one endpoint in $X$, the other in $Y$.

Given two graphs $G$ and $H$, we say that $f : V(H) \rightarrow V(G)$ is a \emph{graph homomorphism} if $f(x)f(y) \in E(G)$ whenever $xy \in E(H)$.
If $f$ is additionally injective, we say that $f$ is an \emph{embedding (of $H$ into $G$)}. Then $H \subseteq G$.

\smallskip

Throughout the paper we will ignore floors and ceilings wherever they do not affect the argument.
The constants in the hierarchies used to state our results are chosen from right to left.
For example, if we claim that a result holds whenever $0<1/n\ll a\ll b\ll c\le 1$ (where $n$ is the order of the graph), then 
there are non-decreasing functions $f:(0,1]\to (0,1]$, $g:(0,1]\to (0,1]$ and $h:(0,1]\to (0,1]$ such that the result holds
for all $0<a,b,c\le 1$ and all $n\in \mathbb{N}$ with $b\le f(c)$, $a\le g(b)$ and $1/n\le h(a)$. 
Note that $a \ll b$ implies that we may assume in the proof that e.g. $a < b$ or $a < b^2$.

Given numbers  $a,b,c$, we write $a=b\pm c$ to mean $a \in [b-c,b+c]$.

\subsubsection{Named graphs}\label{notation}
Given a graph $H$, the graph $H^r$, called the \emph{$r$th power of $H$}, is obtained from $H$ by adding an edge between every pair of vertices of distance at most $r$ in $H$.
In particular:
\begin{itemize}
\item $P^r_k = P = v_1\ldots v_k$ is an \emph{$r$-path} if $V(P) = \lbrace v_1,\ldots,v_k\rbrace$ and $E(P) = \bigcup_{j \in [r]}\lbrace v_iv_{i+j} : 1 \leq i \leq k-j\rbrace$; and
\item $C^r_k = C = w_1\ldots w_k$ is an \emph{$r$-cycle} if $V(C) = \lbrace w_1,\ldots,w_k\rbrace$ and 
$E(C) = \bigcup_{j \in [r]} \lbrace w_iw_{i+j} : 1 \leq i \leq k\rbrace$, where addition is modulo $k$.
\end{itemize}
Additionally,
\begin{itemize}
 \item $F$ is an \emph{$r$-trail (of length $s$)} if there exists an ordered sequence of not necessarily distinct vertices $v_1,\ldots,v_s$ such that $V(F) = \lbrace v_1,\ldots,v_s\rbrace$ and $E(F) = \bigcup_{j \in [r]}\lbrace v_iv_{i+j} : 1 \leq i \leq s-j\rbrace$.
Observe that $P^r_k$ is an $r$-trail, and $F \cong P^r_s$ if and only if $|F|=s$. 

\item A \emph{$K$-tiling} is a collection of vertex disjoint copies of $K$. If it contains $t$ copies, we denote it by $t\cdot K$. If $H \subseteq G$ is a  $K$-tiling which is also spanning, we say that $H$ is a \emph{$K$-factor} of $G$.
\end{itemize}
Define
\begin{itemize}
\item  $Z^r_\ell$ to be the graph with vertex set $[\ell]\times[r]$ in which $(i,j)(i',j')$ is an edge whenever (i) $|i-i'| \leq 1$ and $j \neq j'$ and when (ii) $i=\ell$, $i'=1$ and $j \neq j'$.
\end{itemize}
Thus, $Z^r _\ell$ is obtained from a cycle on $\ell$ vertices by replacing each vertex with a clique on $r$ vertices and replacing every edge with a complete bipartite graph minus a certain perfect matching.
As indicated in Section~\ref{sketch}, $Z^{2r}_\ell$  will be used in the proof of Theorem~\ref{main} as a framework for embedding (most of) $H$ into $G$.
 Note that B\"ottcher, Schacht and Taraz~\cite{bot} used a very similar structure in their proof of the Bandwidth theorem.

\smallskip

Observe that
\begin{equation}\label{CZ}
2\ell\cdot K_r \subseteq C^{r-1}_{2r\ell} \subseteq Z^r_{2\ell} \subseteq C^{2r-1}_{2r\ell} \subseteq Z^{2r}_{\ell},
\end{equation}
and the lexicographic ordering of $V(Z^{r}_\ell)$ (i.e.~$(1,1)(1,2),\ldots,(1,r),(2,1),\ldots,(\ell,r)$) is an $(r-1)$-cycle ordering of $C^{r-1}_{r\ell}$.\COMMENT{AT: is there a better way to express this last statement?}

\medskip
Given $A,B \subseteq V(G)$ and $x_1,\ldots,x_k \in V(G)$, when we say that e.g. $ABx_1\ldots x_k$ is an $r$-path (respectively $r$-trail, $r$-cycle), we mean that any ordering $a_1,\ldots,a_{|A|}$ of $A$ and any ordering $b_1,\ldots,b_{|B|}$ of $B$ are such that $a_1\ldots a_{|A|}b_1\ldots b_{|B|} x_1 \ldots x_k$ is an $r$-path (respectively $r$-trail, $r$-cycle). An $r$-path (respectively $r$-trail, $r$-cycle), $Ax_1\ldots x_kB$ or $x_1\ldots x_kAB$ is defined analogously.

Suppose $X$ and $Y$ are disjoint sets of vertices of size $r$.
We say that an $r$-path $P$ is \emph{between} $X$ and $Y$ if $P = Xx_1\ldots x_kY$ for some vertices $x_1,\ldots,x_k$.
Observe that $P[X],P[Y] \cong K_r$.
Further, $P$ \emph{avoids} a set $W \subseteq V(G)$ if $V(P)\cap W = \emptyset$.


\subsection{Properties of locally dense graphs}

In this section we prove some simple properties of locally dense graphs $G$: that $G$ induced on a large vertex subset is still locally dense; after removing a small set of vertices, $G$ is still locally dense; and $G$ contains many copies of cliques of a fixed size which additionally have a large common neighbourhood.

A fact that we shall use often throughout the paper is that if $0<\rho<\rho'$ and $0<d'<d$, then a $(\rho,d)$-dense graph is also $(\rho',d')$-dense.

\begin{lemma}\label{nbrhd}
Let $r,n \in \mathbb{N}$ and $0 < 1/n \ll \rho \ll d,1/r$, and $0 < d, \alpha < 1$.
 Let $G$ be a $(\rho,d)$-dense graph on $n$ vertices and let $U \subseteq V(G)$ where $|U|=\alpha n$.
Then
\begin{itemize}
\item[(i)] $G[U]$ is $(\rho/\alpha^2,d)$-dense;
\item[(ii)] $G\setminus U$ is $(\rho/(1-\alpha)^2,d)$-dense;
\item[(iii)] $G$ contains at least $dn/2$ vertices of degree at least $dn/2$;
\item[(iv)] $G$ contains at least $(d/2)^{\binom{r+1}{2}}n^r/r!$ copies of $K_r$, each of which is $d^rn/2^r$-extendable.
\end{itemize}
\end{lemma}

\begin{proof}
The proof of~(i) is clear and (ii) follows immediately from (i). 
For~(iii), let $Y := \lbrace v \in V(G) : d_G(v) \geq dn/2\rbrace$.
Then
$$
2d\binom{n}{2} - 2\rho n^2 \leq 2e(G) \leq (n-|Y|) \frac{dn}{2} + |Y| n
$$
and so
$$
|Y| \geq \frac{dn-2d-4\rho n}{2-d} \geq \frac{dn}{2},
$$
proving~(iii).

It remains to prove~(iv).
We claim that for each $i \leq r$, there is a set $\mathcal{T}_i$ of (ordered) tuples $\V x = (x_1,\ldots,x_i)$ such that $G[\lbrace x_1,\ldots,x_i \rbrace] \cong K_i$ and $d_G(\lbrace x_1,\ldots,x_i\rbrace) \geq d^in/2^i$ for all $\V x \in \mathcal{T}_i$, and $|\mathcal{T}_i| \geq (d/2)^{\binom{i+1}{2}}n^i$.
This will immediately imply (iv) as $\mathcal{T}_r$ gives rise to at least $(d/2)^{\binom{r+1}{2}}n^r/r!$ (unlabelled) copies of $K_r$ each of which is $d^rn/2^r$-extendable.

We will prove this by induction on $i$. Part (iii) implies that $G$ contains a set $\mathcal{T}_1$ of $dn/2$ copies of $K_1$ which are all $dn/2$-extendable.
Suppose that we have obtained $\mathcal{T}_{i-1}$ with the required properties for some $2 \leq i \leq r$.

Fix $\V x = (x_1,\ldots,x_{i-1}) \in \mathcal{T}_{i-1}$.
The graph $G_{\V x} := G[N_G(x_1,\ldots,x_{i-1})]$ induced by its neighbourhood contains at least $d^{i-1}n/2^{i-1}$ vertices and so (i) implies that it is $(2^{2i-2}\rho/d^{2i-2},d)$-dense and hence $(\sqrt{\rho},d)$-dense.
Now, using the fact that $1/n \ll \sqrt{\rho} \ll d,1/r$,~(iii) implies that $G_{\V x}$ contains at least $(d/2)\cdot d^{i-1}n/2^{i-1} = d^in/2^i$ vertices each of degree at least $d^in/2^i$.
Each such vertex $y$ gives rise to an $r$-tuple $\V x(y) := (x_1,\ldots,x_{i-1},y)$.
Certainly $G[\lbrace x_1,\ldots,x_{i-1},y\rbrace] \cong K_i$, and further $d_G(\lbrace x_1,\ldots,x_{i-1},y\rbrace) \geq d^in/2^i$ since $y$ has at least this many neighbours in the common neighbourhood of $\V x$.
Let $\mathcal{T}_i$ be the collection of all these tuples $\V x(y)$ formed from each $\V x \in \mathcal{T}_{i-1}$.
Observe that they have the required properties, and are all distinct, so
$$
|\mathcal{T}_i| \geq d^in/2^i \cdot |\mathcal{T}_{i-1}| \geq (d/2)^{\binom{i}{2}+i}n^i = (d/2)^{\binom{i+1}{2}}n^i.
$$
This completes the proof of the lemma.
\end{proof}

We will need a \emph{Connecting lemma} to find a short $r$-path between two `extendable' copies of $K_{r}$ in a locally dense graph $G$ with $\delta(G) > (1/2+o(1))n$.
\COMMENT{Replaced $\Omega(1)$ with $o(1)$ here to be consistent. I think now with $>$ notation this is ok}
The heart of the proof is the following lemma, which is the only part of the proof of Theorem~\ref{main} which requires $\delta(G) > (1/2+o(1))n$ (elsewhere, linear minimum degree suffices). Somewhat similar lemmas have been used elsewhere in other settings e.g.~\cite{ferb, st1}.

\begin{lemma}\label{ingredient}
Let $0 < 1/n \ll \rho \ll d,\eta, 1/r < 1$ where $n,r \in \mathbb{N}$.
Let $G$ be an $n$-vertex graph and let $U \subseteq V(G)$ be a subset of size $n' \geq \eta n/2$ such that $G[U]$ is $(\rho,d)$-dense and $d_G(x,U) \geq (1/2+\eta)n'$ for all $x \in V(G)$.
Let $X,Y,W$ be pairwise disjoint subsets of $V(G)$ such that $|X|=|Y| = \lceil 4r/\eta \rceil$ and $|W| \leq \eta n'/2$. Then there is $Z \subseteq U$ such that
\begin{itemize}
\item[(i)] $G[Z] \cong K_r$;
\item[(ii)] $Z \cap (X\cup Y \cup W)=\emptyset$;
\item[(iii)] there exist $X' \subseteq X$ and $Y' \subseteq Y$ with $|X'|=|Y'|=r$ such that $N_G(Z) \supseteq X' \cup Y'$.
\end{itemize}
\end{lemma}

\begin{proof}
Let $C := \lceil 4r/\eta \rceil$ and let $U' := U\setminus (X\cup Y \cup W)$.
Then
$$
e_G(U',X\cup Y) \geq (|X|+|Y|)((1/2+\eta)n' - |X|-|Y|-|W|) \geq 2C(1/2+\eta/4)n' \geq (1+\eta/2)Cn'.
$$
Let $U''$ be the collection of those vertices in $U'$ which each have at least $C+r$ neighbours in $X\cup Y$.
Then $(1+\eta/2)Cn' \leq (C+r-1)(n'-|U''|) + 2C|U''|$ and so
$$
|U''| \geq \frac{(1+\eta/2)Cn' - (C+r-1)n'}{C-r+1} \geq \frac{\eta n'}{4},
$$
where the final inequality follows from the fact that $C \geq 4r/\eta$.
There are not more than $2^{2C}$ ways a vertex can attach to $X \cup Y$, so there is $U^* \subseteq U''$ such that $N_G(v,X\cup Y)$ is identical for all $v \in U^*$ and $|U^*| \geq \eta n'/2^{2C+2}$.
Note further that, since each such $v$ has at least $C+r$ neighbours in $X \cup Y$, there is $X' \subseteq X$ and $Y' \subseteq Y$ such that $|X'|=|Y'|=r$ and $X' \cup Y' \subseteq N_G(U^*)$.
Lemma~\ref{nbrhd}(i) now implies that $G[U^*]$ is $(2^{4C+4}\rho/\eta^2,d)$-dense and hence $(\sqrt{\rho},d)$-dense.
But then by~Lemma~~\ref{nbrhd}(iv) there is $Z \subseteq U^*$ which spans a $K_{r}$.
The desired properties~(i)--(iii) are immediate.
\end{proof}
As well as being applied in the proof of the Connecting lemma below, Lemma~\ref{ingredient} is also a key tool in the proof of Theorem~\ref{rcycle} in Section~\ref{cyclesec}, which in turn is a crucial tool for the proof of Theorem~\ref{main}.

\begin{lemma}[Connecting lemma]\label{connect}
Let $0 < 1/n \ll \rho \ll d,\eta \leq 1/r$ where $r \in \mathbb{N}$ and let $G$ be a $(\rho,d)$-dense graph on $n$ vertices with $\delta(G)\geq (1/2+\eta)n$.
Let $W,X,Y$ be subsets of $V(G)$ such that $|W| \leq \eta n/4$ and $X,Y$ induce $r$-cliques in $G$, and each one either
\begin{itemize}
\item lies in a copy of $K_{\lceil 9 r/\eta\rceil}$ which is disjoint from $W$; or
\item is $\eta n$-extendable.
\end{itemize}
Then $G$ contains a copy of $P^{r}_{3r} = x_1\ldots x_{3r}$ avoiding $W$
such that $Xx_1\ldots x_{3r}$ induces a copy of $P^{r}_{4r}$, and $x_1\ldots x_{3r}Y$ induces a copy of $P^{r}_{4r}$.
\end{lemma}
The Connecting lemma will ensure that the reduced graph $R$ of a graph $G$ (as in Theorem~\ref{main}) is `highly connected'. This property will be exploited when embedding a part of $H$ into $G$ so as to cover all of the exceptional set $V_0$ (specifically, we make use of Lemma~\ref{connect} in Section~\ref{special}).

\begin{proof}
Suppose that $X$ is $\eta n$-extendable.
Let $C := \lceil 9r/\eta \rceil$ and $c := \lceil 4r/\eta\rceil$, and also let $G_X := G[N_G(X)\setminus W]$. Then Lemma~\ref{nbrhd}(i)  implies that $G_X$ is a $(16\rho/(9\eta^2),d)$-dense graph on at least $3\eta n/4$ vertices.
But $4/(3\eta n) \ll 16\rho/(9\eta^2) \ll d, 1/C$, so Lemma~\ref{nbrhd}(iv) implies that $G_X$ contains a copy of $K_{C}$.
Therefore $X$ lies in a copy of $K_{C+r}$ which does not intersect $W$.

This implies that we may assume  both $X,Y$ lie in a copy of $K_{C}$ which does not intersect $W$.
Let $X^*$ be the vertex set of the $K_{C}$ containing $X$ and define $Y^*$ analogously for $Y$.
Choose $X' \subseteq X^*$ of size $c$ which is disjoint from $X$.
Since $|Y^*|-|Y|-|X|-|X'| = C - 2r - c \geq c$,
we can choose $Y' \subseteq Y^*$ of size $c$ which is disjoint from $X,Y,X'$.
Apply Lemma~\ref{ingredient} with $n,r,\eta,V(G),X',Y',X \cup Y \cup W$ playing the roles of $n,r,\eta,U,X,Y,W$ to obtain $Z \subseteq V(G)$ which induces a copy of $K_{r}$; is disjoint from $X' \cup Y'\cup X \cup Y \cup W$, and there exist $X'' \subseteq X'$ and $Y'' \subseteq Y'$ such that $|X''|=|Y''|=r$ and $X'' \cup Y'' \subseteq N_G(Z)$.
Notice that, by construction, each of $X \cup X''$, $X'' \cup Z$, $Z \cup Y''$ and $Y'' \cup Y$ induce cliques, and the overlap of each consecutive pair induces a clique of size at least $r$. Further, none of these sets intersect with $W$.
Thus $XX''ZY''Y$ induces an $r$-path. Thus there is an $r$-path with vertex set $X''\cup Z \cup Y''$ (of length $3r$) which has the required property.
\end{proof}

\section{The Regularity and Blow-up lemmas and associated tools}\label{secrl}

\subsection{Regularity}
We will apply Szemer\'edi's Regularity lemma in the proof  of Theorem~\ref{main}. For this we need the following definitions. Given a bipartite graph $G$ with vertex classes $A$ and $B$ and parameters $\eps,\delta \in (0,1)$,
\begin{itemize}
\item let $d_G(A,B) := \frac{e_G(A,B)}{|A||B|}$ be the \emph{density} of $G$; and say that $G$ is
\item \emph{$\eps$-regular} if, for every $X \subseteq A$ and $Y \subseteq B$ with $|X| \geq \eps|A|$ and $|Y| \geq \eps|B|$ we have that $|d_G(A,B)-d_G(X,Y)| \leq \eps$;
\item \emph{$(\eps,\delta)$-regular} if $G$ is $\eps$-regular and additionally $d_G(A,B) \geq \delta$; 
\item \emph{$(\eps,\delta)$-superregular} if $G$ is $(\eps,\delta)$-regular and additionally $d_G(a,B) \geq \delta|B|$ for every $a \in A$ and $d_G(b,A) \geq \delta|A|$ for every $b \in B$.%
\end{itemize} 

It will be convenient to use the degree form of the Regularity lemma; this can be derived from the standard version~\cite{reglem}.

\begin{lemma}[Degree form of the Regularity lemma]\label{reg}
For all $\eps \in (0,1)$ and $M' \in \mathbb{N}$, there exist $M,n_0\in \mathbb{N}$ such that the following holds for all graphs $G$ on $n \geq n_0$ vertices and $\delta \in (0,1)$.
There is a partition $V(G) = V_0 \cup V_1 \cup \ldots \cup V_L$ and a spanning subgraph $G' \subseteq G$ such that
\begin{itemize}
\item[(i)] $M' \leq L \leq M$;
\item[(ii)] $|V_0| \leq \eps n$;
\item[(iii)] $|V_1|=\ldots = |V_L| =: m$;
\item[(iv)] $d_{G'}(x) \geq d_G(x)-(\delta+\eps)n$ for all $x \in V(G)$;
\item[(v)] for all $i \in [L]$ the graph $G'[V_i]$ is empty;
\item[(vi)] for all $i \in [L]$ the graph $G'[V_i,V_j]$ is either empty or $(\eps,\delta)$-regular.
\end{itemize}
\end{lemma}

We call $V_1,\ldots,V_L$ the \emph{clusters} of $G$ and the vertices in $V_0$ the \emph{exceptional vertices}.
The graph $G'$ is the \emph{pure graph}.
Note that the $(\eps,\delta)$-regular pairs may have very different densities.
The \emph{reduced graph} $R$ of $G$ with parameters $\eps,\delta$ and $M'$ has vertex set $[L]$ and contains $ij$ as an edge precisely when $G'[V_i,V_j]$ is $(\eps,\delta)$-regular.

The next lemma states that the reduced graph $R$ of a locally dense graph $G$ is still locally dense (with worse parameters), and further $R$ inherits the minimum degree of $G$.

\begin{lemma}\label{inherit}
Let $0<1/n \ll  1/M' \ll \eps\ll\delta  \ll d \leq 1$; $1/M' \ll \rho \ll d$; $\delta \ll \eta$. Define $\rho ^*:= \max \{ 3 \rho, 3 \delta \}$.
Let $G$ be a $(\rho,d)$-dense graph of order $n$ with $\delta(G) \geq (1/2+\eta)n$.
Apply Lemma~\ref{reg} with parameters $\eps,\delta$ and $M'$ to obtain a pure graph $G'$ and a reduced graph $R$ of $G$ with $V(R)=[L]$.
Then $R$ is $(\rho ^*,d)$-dense with $\delta(R) \geq (1/2+\eta/2)L$.
\end{lemma}

\begin{proof}
Here~(i)--(vi) will refer to the conclusions of Lemma~\ref{reg}.
Parts~(ii) and~(iii) imply that
\begin{equation}\label{boring}
(1-\eps)n \leq mL \leq n.
\end{equation}
Let $X \subseteq [L]$ and let $Y := \bigcup_{i \in X}V_i \subseteq V(G)$. So $|Y|=m|X|$.
Then
$$
d\binom{|Y|}{2} - \rho n^2 - |Y|(\delta+\eps)n \leq e(G[Y]) - |Y|(\delta+\eps)n \stackrel{(iv)}{\leq} e(G'[Y]) \stackrel{(v)}{\leq} e(R[X])\cdot m^2
$$
and so, dividing by $m^2$,
$$
e(R[X]) \stackrel{(\ref{boring})}{\geq} d\cdot\frac{|X|^2 - \frac{|X|}{m}}{2} - \rho \left(\frac{L}{1-\eps}\right)^2 - |X|(\delta+\eps)\frac{L}{1-\eps} \geq d\binom{|X|}{2} - \rho ^* L^2,
$$
as required.

Let $i \in [L]$ and $x_i \in V_i$. Then $d_{G'}(x_i) \geq d_G(x_i) - (\delta+\eps)n \geq (1/2+\eta-\delta-\eps)n$ by~(iv).
The number of clusters $V_k$ of $G$ containing some $y \in N_{G'}(x_i)$ is therefore at least
$$
\frac{(1/2 + \eta-\delta-\eps)n - |V_0|}{m} \geq \frac{(1/2+\eta/2)n}{m} \geq (1/2+\eta/2)L.
$$
But then~(vi) implies that $i$ is adjacent to each of the vertices corresponding to these clusters in $R$. So $d_R(i) \geq (1/2+\eta/2) L$, as required. 
\end{proof}
Note that in the case when $\rho \ll \delta$ in Lemma~\ref{inherit}, $R$ only inherits the property being locally dense with a significantly worse parameter playing the role of $\rho$.
That is, now $R$ is $(3\delta, d)$-dense rather than $(\rho,d)$-dense. 

The next well-known proposition states that (super)regular pairs are robust in the sense of adding or removing a small number of vertices. This version appears as Proposition~8 in~\cite{3partite}.

\begin{proposition}\label{newsuperslice}
Let $G$ be a graph with $A,B \subseteq V(G)$ disjoint.
Suppose that $G[A,B]$ is $(\eps,\delta)$-regular and let $A',B' \subseteq V(G)$ be disjoint such that $|A \triangle A'| \leq \alpha|A'|$ and $|B \triangle B'| \leq \alpha|B'|$ for some $0 \leq \alpha < 1$.
Then $G[A',B']$ is $(\eps',\delta')$-regular, with
$$
\eps' := \eps + 6\sqrt{\alpha} \ \ \text{ and }\ \ \delta' := \delta-4\alpha.
$$
If, moreover, $G[A,B]$ is $(\eps,\delta)$-superregular and each vertex $x \in A'$ has at least $\delta'|B'|$ neighbours in $B'$ and each vertex $x \in B'$ has at least $\delta'|A'|$ neighbours in $A'$, then $G[A',B']$ is $(\eps',\delta')$-superregular with $\eps'$ and $\delta'$ as above.
\end{proposition}

The following lemma is well known in several variations.
The version here 
follows immediately from \cite[Lemma 4.6]{starxiv}.
\begin{lemma}\label{superslice}
Let $L \in \mathbb{N}$ and suppose that $0 < 1/m  \ll \eps \ll \delta, 1/\Delta, 1/L \leq 1$.
Let $R$ be a graph with $V(R) = [L]$ and $\Delta (R)\leq \Delta$.
Let $G$ be a graph with vertex partition $V_1, \ldots, V_{L}$ such that $|V_i|=m$ for all $1 \leq i \leq L$, and in which $G[V_i,V_j]$ is $(\eps,\delta)$-regular whenever $ij \in E(R)$.
Then for each $i \in V(R)$, $V_i$ contains a subset $V_i'$ of size $(1-\sqrt{\eps})m$ such that for every edge $ij$ of $R$, the graph $G[V_i',V_j']$ is $(4\sqrt{\eps},\delta/2)$-superregular.
\end{lemma}
\COMMENT{AT: modified the statement of this lemma, as it didn't quite work correctly for us before (didn't have $L$ huge in our application)}

\subsection{Embedding lemmas}

The next lemma is similar to a  \emph{Partial embedding lemma} from~\cite[Lemma 10]{3partite} which in turn is similar to an embedding lemma 
due to Chv\'atal, R\"odl, Szemer\'edi and Trotter~\cite{trot}.
Given a homomorphism from a graph $H$ into the reduced graph $R$ of $G$  such that every pre-image is small, the lemma yields an embedding of some vertices of $H$ into $G$, while finding large candidate sets for the remaining vertices.
Further (deviating from~\cite{3partite}), we would like to ensure that certain vertices of $H$ are embedded into given target sets of large size.

\begin{lemma}[Embedding lemma with target sets]\label{target}
Let $0 < 1/n  \ll 1/L \ll \eps \ll c \ll \delta  \ll 1/\Delta$ where $n,L \in \mathbb{N}$.
Let $G$ be an $n$-vertex graph, $R$ an $L$-vertex graph and $H$ a graph on at most $\eps n$ vertices such that
\begin{itemize}
\item $G$ has partition $\lbrace V_a : a \in V(R)\rbrace$ where $|V_a| \geq m \geq (1-\eps)n/L$ for all $a \in V(R)$ and $G[V_a,V_{a'}]$ is $(\eps,\delta)$-regular whenever $aa' \in E(R)$.
\item $\Delta(H)\leq\Delta$ and there is a graph homomorphism $\phi: V(H)\rightarrow V(R)$ such that $|\phi^{-1}(a)|\leq 2\eps m$ for all $a \in V(R)$.
\item Let $X \cup Y$ be a partition of $V(H)$ and suppose that there is 
$W \subseteq X$ such that for each $w \in W$, there is a set $S_w \subseteq V_{\phi(w)}$ with $|S_w| \geq cm$.
\end{itemize}
Then there is an embedding $f$ of $H[X]$ into $G$ such that
\begin{itemize}
\item[(i)] $f(x) \in V_{\phi(x)}$ for all $x \in X$;
\item[(ii)] $f(w) \in S_w$ for all $w \in W$;
\item[(iii)] for all $y \in Y$ there exists $C_y \subseteq V_{\phi(y)}\setminus f(X)$ such that $C_y \subseteq N_G(f(x))$ for all $x \in N_H(y) \cap X$, and $|C_y| \geq cm$.
\end{itemize} 
\end{lemma}%
Since the proof of Lemma~\ref{target} is essentially identical to that of Lemma~10 from ~\cite{3partite}, we omit the proof.

We will also use the Blow-up lemma of Koml\'os, S\'ark\"ozy and Szemer\'edi~\cite{kss2}, which states that, for the purposes of embedding a spanning $k$-partite graph $H$ of bounded degree, a graph $G$ with a vertex partition into $k$ classes, each pair of which is superregular, in fact behaves like a complete $k$-partite graph.
Further, as in Lemma~\ref{target}, one can ensure that a small fraction of the vertices of $H$ are embedded into some given target sets.

\begin{lemma}[Blow-up lemma~\cite{kss2}]\label{blowlem}
For every $d,\Delta,c > 0$ and $k \in \mathbb{N}$ there exist constants $\eps_0$ and $\alpha$ such that the following holds.
Let $n_1, \ldots, n_k$ be positive integers,  $0 < \eps < \eps_0$, and $G$ be a $k$-partite graph with vertex classes $V_1, \ldots, V_k$ where $|V_i|=n_i$ for $i \in [k]$.
Let $J$ be a graph on vertex set $[k]$ such that $G[V_i,V_j]$ is $(\eps,d)$-superregular whenever $ij \in E(J)$.
Suppose that $H$ is a $k$-partite graph with vertex classes $W_1, \ldots, W_k$ of size  at most $n_1, \ldots, n_k$ respectively with $\Delta(H) \leq \Delta$.
Suppose further that there exists a graph homomorphism $\phi : V(H) \rightarrow V(J)$ such that $|\phi^{-1}(i)| \leq n_i$ for every $i \in [k]$.
Moreover, suppose that in each class $W_i$ there is a set of at most $\alpha n_i$ special vertices $y$, each of them equipped with a set $S_y \subseteq V_i$ with $|S_y| \geq cn_i$.
Then there is an embedding of $H$ into $G$ such that every special vertex $y$ is mapped to a vertex in $S_y$.
\end{lemma}


\section{Finding the power of a Hamilton cycle}\label{cyclesec}
The next result
states that for every $r \in \mathbb{N}$, every large locally dense $n$-vertex graph $G$  with minimum degree at least $(1/2+o(1))n$ contains the $r$th power of a Hamilton cycle.
This is a very special case of our main result, Theorem~\ref{main}.

\begin{theorem}\label{rcycle}
For all $r,s\in \mathbb{N}$ and $d,\eta>0$, there exist $\rho,n_0>0$ such that every $(\rho,d)$-dense graph $G$ on $n \geq n_0$ vertices with $\delta(G)\geq (1/2+\eta)n$ contains the $r$th power of a  Hamilton cycle.
In fact, for every $n'\in \mathbb N$ such that $n-s\leq n' \leq n$, $G$ contains the $r$th power of a cycle covering precisely $n'$ vertices.
\end{theorem}
Note that Theorem~\ref{rcycle} is an important tool in the  proof of Theorem~\ref{main}, in the same way that (an approximate version of) the result in~\cite{kss} was used in the proof of Theorem~\ref{bst}.
Indeed, Theorem~\ref{rcycle}  ensures that the reduced graph $R$ of a graph $G$ (as in Theorem~\ref{main}) will contain a spanning $(4r-1)$-cycle.
By (\ref{CZ}) this implies $R$ contains a spanning copy of $Z^{2r} _{\ell}$.
\COMMENT{Actually in this application of Theorem~\ref{rcycle}, we will instead obtain a spanning $r^*$-cycle in $R$ where $r^*$ is significantly bigger than $4r$. This ensures we obtain an extra useful property, in addition to the main desired property that 
$Z^{2r} _{\ell}$ lies in $R$.}
As outlined in Section~\ref{sketch}, this copy of $Z^{2r} _{\ell}$ will be used as a `guide' for embedding $H$ into $G$.

We remark  that one can give a significantly shorter proof of Theorem~\ref{rcycle} if one only seeks the $r$th power of a  cycle covering (say) at least  $(1-\eta )n $ vertices in $G$. However, for our application to Theorem~\ref{main} we (rather subtly) require that
we have a $(4r-1)$-cycle in $R$ covering all but a very small number of vertices (much fewer than $\rho|R|$ vertices in $R$ can be left uncovered). So such a weaker version of Theorem~\ref{rcycle} is not sufficient.

The proof of Theorem~\ref{rcycle} is an application of the \emph{Connecting--Absorbing method}, a technique first developed by R\"odl,  Ruci\'nski and  Szemer\'edi~\cite{rrs2}.
The first step in the proof is to find a short \emph{absorbing} $2r$-path $P_{abs}$ in $G$ which has the property that $V(P_{abs})\cup Z$ spans an $r$-path in $G$ (with the same start- and endpoints as $P_{abs}$) for any very small set of vertices $Z$.
We then reserve a small pot of vertices $V'$ (known as  a \emph{reservoir}), that will allow us to connect up pairs of paths into longer paths.
Next we (via an application of the Regularity lemma) find a collection $\mathcal P$ of a constant number of vertex-disjoint $2C$-paths that together cover almost all of the remaining vertices in $G$ (here $C$ is chosen to be significantly bigger than $r$).
Using vertices from the reservoir, we are then able to connect together all the paths in $\mathcal P$ together with $P_{abs}$ to form a single $r$-cycle  covering almost all the vertices in $G$. The remaining uncovered vertices in $G$ are absorbed 
by $P_{abs}$ to obtain the $r$th power of a Hamilton cycle.

\smallskip

{\noindent \it Proof of Theorem~\ref{rcycle}.}
Note that if $n$ is sufficiently large then any $n'$-vertex induced subgraph $G'$ of an $n$-vertex graph $G$ as in the theorem must be $(2\rho,d)$-dense with  $\delta (G')\geq (1/2+\eta/2)n'$.
So as the $r$th power of a Hamilton cycle in $G'$ corresponds to an $r$-cycle  of length $n'$ in $G$,
it suffices to prove the first part of the statement of the theorem.
 
Further, it suffices to prove the theorem under the additional assumption that $d \ll \eta, 1/r$.
Define constants $ \rho, \eps, \delta, d_1, \eta _0,\eta _1,\eta_2, \eta _3>0$ and $M' \in \mathbb N$, and apply the Regularity lemma (Lemma~\ref{reg}) 
with inputs $\eps$ and $M'$ to obtain some $M=M(\eps,M')$ so that we have
\begin{align}\label{hier}
0<  1/M \leq 1/M' \ll \eps \ll \delta \ll \rho   \ll \eta _3  \ll \eta _2 \ll \eta _1 \ll \eta _0 \ll d_1 \ll d \ll \eta, 1/r.
\end{align}
Let $n$ be sufficiently large, and consider any $n$-vertex graph $G$ that is $(\rho,d)$-dense  with $\delta(G)\geq (1/2+\eta)n$.

Our initial aim is to construct a small absorbing $2r$-path $P_{abs}$. The next claim provides the building blocks for this absorbing path.
\begin{claim}\label{block}
There exists a collection $\mathcal K$ of at most $\eta  _0n/8r$ vertex-disjoint copies of $K_{2r}$ in $G$ such that:
\begin{itemize}
\item[(i)] Each $K \in \mathcal K$ is $d_1n$-extendable in $G$.
\item[(ii)] Given any vertex $x\in V(G)$, there are at least $2\eta _2 n$ copies $K$ of $K_{2r}$ in $\mathcal K$ so that $V(K) \subseteq N_G(x)$.
\end{itemize}
\end{claim}
\begin{claimproof}
Let $\mathcal C$ denote the set of all copies of $K_{2r}$  that are $d_1n$-extendable in $G$. So certainly $|\mathcal C|\leq n^{2r}$.
Consider any $x \in V(G)$. Since $d_G(x)\geq n/2$, Lemma~\ref{nbrhd}(i) implies that $G[N(x)]$ is $(4 \rho, d)$-dense. Thus, Lemma~\ref{nbrhd}(iv) implies that
there are at least $(d/2)^{\binom{2r+1}{2}} (n/2)^{2r}/(2r)!$ copies $K$ of $K_{2r}$ in $\mathcal C$ so that $V(K) \subseteq N_G(x)$. 
(Here we use the property that $d^{2r}/2^{2r} \geq d_1$ by (\ref{hier}).) Let $L_x$ denote the set of these copies of $K_{2r}$.

Let $\mathcal C_p$ be obtained from $\mathcal C$ by selecting each $K \in \mathcal C$ independently with probability
$$p:= \frac{\eta _1}{n^{2r-1}}.$$
Hence, 
$$\mathbb E (|\mathcal C_p|)\leq\eta _1 n \ \ \text{ and } \ \ \mathbb E( |\mathcal C_p \cap L_x |) \geq (d/2)^{\binom{2r+1}{2}} \frac{ (n/2)^{2r}}{(2r)! } \times  \frac{\eta _1}{n^{2r-1}} \stackrel{(\ref{hier})}{\geq} 2 d_1 \eta _1 n $$
for each $x \in V(G)$.
Thus, a Chernoff bound implies that, with high probability,
\begin{align}\label{nicebounds}
|\mathcal C_p| \leq 2\eta _1 n \ \ \text{ and } \ \ |\mathcal C_p \cap L_x |\geq  d_1 \eta _1 n 
\end{align}
for all $x \in V(G)$.
Let $Y$ denote the number of pairs of copies of $K_{2r}$ from $\mathcal C_p$ that share at least one vertex.
Then
$$\mathbb E( Y) \leq p^2 \binom{n}{2r} 2r \binom{n}{2r-1} \leq \eta ^2 _1 n.$$ 
By Markov's inequality the probability that $|Y|\leq 2\eta ^2 _1 n$ is at least $1/2$. Therefore, there is a choice of $\mathcal C_p$ such that this condition holds together with
(\ref{nicebounds}). Fix such a choice of $\mathcal C_p$; then for each intersecting pair of cliques in $\mathcal C_p$, remove one of them to obtain a new collection $\mathcal K$.
Note that the definition of $\mathcal C_p$ and (\ref{nicebounds}) implies that $\mathcal K$ is a collection of at most $\eta _0 n/8r$ vertex-disjoint copies of $K_{2r}$ in $G$.
Further, since $d_1 \eta _1 n -  2\eta ^2 _1 n\geq d_1 \eta _1 n/2 \geq 2 \eta _2 n$, we see that (ii) is satisfied, as desired.
\end{claimproof}

\smallskip

With Claim~\ref{block} at hand, it is straightforward to obtain our desired absorbing $2r$-path $P_{abs}$.

\begin{claim}\label{pabs}
$G$ contains a $2r$-path $P_{abs}$ on at most $\eta _0 n$ vertices such that the following conditions hold. 
\begin{itemize}
\item[(i)] Both the set of the first and last $2r$ vertices on $P_{abs}$ induce $K_{2r}$s in $G$ that are $d_1n$-extendable. (Denote these sets by $S$ and $E$ respectively.)
\item[(ii)] Given any set $Z\subseteq V(G)\setminus V(P_{abs})$ of size at most $\eta _2 n$, there is an $r$-path $P$ in $G$ with vertex set $V(P_{abs}) \cup Z$ whose first $2r$ vertices are the elements of $S$ (ordered as in $P_{abs}$) and
the last $2r$ vertices are the elements of $E$ (ordered as in $P_{abs}$).
\end{itemize}
\end{claim}
\begin{claimproof}
Let $\mathcal K$ be as in Claim~\ref{block}, and enumerate its elements by $K^1, \dots, K^t$ (so $t \leq \eta  _0n/8r$).
Apply Lemma~\ref{connect} to $G$ with $d_1,2r, V(K^1), V(K^2), V(\mathcal K)$ playing the roles of $\eta, r, X,Y,W$. (Note we can indeed apply this lemma by Claim~\ref{block}(i) and as $|V(\mathcal K)|\leq d_1n/4$.)
We thus obtain a copy $P_1=x^1_1\dots x^1_{6r}$ of $P^{2r}_{6r}$ in $G$ avoiding $V(\mathcal K)$ such that $V(K^1)x^1_1\dots x^1_{6r}$ and $x^1_1\dots x^1_{6r}V(K^2)$ both induce copies of $P^{2r}_{8r}$.
Repeating this process iteratively we obtain a collection  $P_1, \dots , P_{t-1}$ of vertex-disjoint copies of $P^{2r}_{6r}$ in $G$ so that
$V(K^i)x^i_1\dots x^i_{6r} V(K^{i+1})$ induces a copy of $P^{2r}_{8r}$ in $G$ for each $1 \leq i \leq t-1$. (Here we have written $P_i=x^i_1\dots x^i_{6r}$.)
Note that to ensure the $P_i$s are vertex-disjoint, at every step we update $W$; so at step $i$, $W$ contains $V(\mathcal K)$ and the vertices from $P_1,\dots, P_{i-1}$ (so $|W|\leq d_1n/4$). 

Let $P_{abs}$ denote the $2r$-path
obtained by the following concatenation:
$$P_{abs}:=V(K^1)P_1V(K^2)P_2V(K^3)\dots V(K^{t-1})P_{t-1}V(K^t).$$

Notice that $P_{abs}$ contains $(t-1)8r+2r \leq 8rt \leq \eta _0 n$ vertices. Further (i) follows since both $K^1$ and $K^{t}$ are $d_1n$-extendable in $G$ by definition of $\mathcal K$.
 Consider any set $Z=\{z_1,\dots, z_{\ell}\}\subseteq V(G)\setminus V(P_{abs})$ of size at most $\eta _2 n$. For each $1\leq i \leq \ell$, by Claim~\ref{block}(ii), there are at least $\eta _2 n$ choices for $j_i$ such that:
\begin{itemize}
\item $2\leq j_i \leq t-1$;
\item $V(K^{j_i})\subseteq N_G (z_i)$.
\end{itemize}
In particular, writing $V(K^{j_i})=\{y_1,\dots, y_{2r}\}$, notice that 
\begin{align}\label{con}
P_{j_i-1}y_1\dots y_r z_i y_{r+1}\dots y_{2r} P_{j_i}
\end{align}
is an $r$-path in $G$.

Since we have at least $\eta _2 n$ choices, we may define $j_1,j_2,\dots, j_{\ell}$ to be distinct.
 We can then insert each $z_i$ into $P_{abs}$ as indicated by~(\ref{con}) to obtain the  desired $r$-path $P$ on $V(P_{abs}) \cup Z$.
\end{claimproof}

\medskip

Let $S$ be as in Claim~\ref{pabs}. Then $|N_G (S) \setminus V(P_{abs})| \geq d_1 n/2$. Lemma~\ref{nbrhd}(i) implies that $G_S:=G[N_G(S)\setminus V(P_{abs})]$ is $(4\rho/d ^2_1,d)$-dense and therefore 
$(\rho ^{1/2},d)$-dense.
Set 
\begin{align}\label{Cdef1}
C:= \lceil 4r/\eta _3 \rceil.
\end{align}
Note that $ \rho ^{1/2} \ll d , 1/C$. Thus, Lemma~\ref{nbrhd}(iv) implies that $G_S$ contains a copy $K^S_{2C+1}$ of $K_{2C+1}$. Similarly, we find a copy $K^E _{2C+1}$ of $K_{2C+1}$ in $G$ that is disjoint from
$K^S_{2C+1}$ and $P_{abs}$ so that $V(K^E_{2C+1}) \subseteq N_G (E)$. We will view both $K^S_{2C+1}$ and $K^E _{2C+1}$ as $2C$-paths of length $2C+1$.

\smallskip

Set $G_0:=G\setminus (V(P_{abs}) \cup V(K^S_{2C+1}) \cup V(K^E_{2C+1}) )$. Certainly, $|G_0|\geq (1-2\eta _0)n$ and 
\begin{align*}
d_G (x,V(G_0))\geq (1/2+3\eta/4)n \text{ for all } x \in V(G).
\end{align*}

By selecting vertices randomly (and
applying a Chernoff bound), one can obtain a set $V'\subseteq V(G_0)$ of $n':=\eta _3 n$ vertices such that 
\begin{align}\label{degV'}
d_G (x,V')\geq (1/2+\eta/2)n' \text{ for all } x \in V(G).
\end{align}
Set $G_1:=G[V']$ and $G_2:=G_0\setminus V'$. Lemma~\ref{nbrhd}(i) implies that $G_1$ is $(\rho/\eta _3^2,d)$-dense and thus, $(\rho ^{1/2},d)$-dense.
Similarly $G_2$ is $(2\rho,d)$-dense. 

Apply Lemma~\ref{reg} to  $G_2$ with parameters $\eps, \delta$ and $M'$ to obtain a partition $V_0,V_1,\dots, V_\ell$ of $V(G_2)$, pure graph $G'_2$ and the reduced graph $R$ of $G_2$.
Here $V_0$ is the exceptional set on at most $\eps n$ vertices and $M'\leq \ell \leq M$.
Set $m:=|V_1|=\dots=|V_\ell|$.
Then Lemma~\ref{inherit} implies that $R$ is $(6\rho,d)$-dense. In particular, Lemma~\ref{nbrhd}(i) implies that $R'$ is $(6\rho/\eta _3 ^2, d)$-dense for any $R'\subseteq R$ on $\eta _3\ell$ vertices.

Note that $1/\ell \ll 6\rho/\eta _3 ^2 \ll d ,1/C$. Thus, Lemma~\ref{nbrhd}(iv) implies that every $R'\subseteq R$ on $\eta _3 \ell$ vertices contains a copy of $K_{2C+1}$. 
In particular, $R$ contains a $K_{2C+1}$-tiling $\mathcal T$
covering all but at most $\eta _3 \ell $ vertices.

Consider any copy $K$ of $K_{2C+1}$ in $\mathcal T$. The vertices of $K$ correspond to clusters $V_{i_1},\dots, V_{i_{2C+1}}$ in $G_2$; let $G_K$ denote the  subgraph of $G'_2$ induced by the vertices in these clusters combined.
Every tuple $(V_{i_j},V_{i_k})$ of such clusters  forms an $\eps$-regular pair of density at least $\delta$ in $G_K$.
Moreover, Lemma~\ref{superslice} implies that for each such cluster $V_{i_j}$ there is a subset $V'_{i_j}\subseteq V_{i_j}$ of size $(1-\eps ^{1/2})m$ so that $(V'_{i_j},V'_{i_k})$ forms an $(4\eps ^{1/2},\delta/2)$-superregular pair in $G_K$ (for each $1 \leq j \not = k \leq 2C+1$).
The Blow-up lemma (Lemma~\ref{blowlem}) now implies that $G_K$ contains a $2C$-path  covering all but at most $(2C+1)\eps ^{1/2}m $ vertices in $G_K$.

Overall, this implies that $G_2$ contains a collection $\mathcal P$ of at most $\ell/(2C+1) \leq M$ vertex-disjoint $2C$-paths, that together cover all but at most 
\begin{align}
\label{cover}\left ( (2C+1)\eps ^{1/2}m \times \frac{\ell}{2C+1} \right )+ \left ({\eta _3 \ell} \times m \right )+|V_0| \leq \eps ^{1/2} n +\eta _3 n+\eps n\stackrel{(\ref{hier})}{\leq} {2 \eta _3 n}
\end{align}
vertices in $G_2$.

We will now use vertices in $G_1$ to connect together all of the $2C$-paths in $\mathcal P\cup \{K^S_{2C+1} , K^E_{2C+1}\}$ to obtain an $r$-path in $G$ whose  first $2C+1$ vertices are the vertices of $K^E_{2C+1}$ and whose last 
$2C+1$ vertices are the vertices of $K^S_{2C+1}$.
Note that we will have to reorder some of the vertices in the $2C$-paths in $\mathcal P$, so that is one reason why we `drop' from $2C$-paths to an $r$-cycle.
Label the $2C$-paths in $\mathcal P\cup \{K^S_{2C+1} , K^E_{2C+1}\}$ by $P_1, \dots, P_t$, where $P_1:=K^E_{2C+1}$ and $P_t := K^S_{2C+1}$.
In particular, note  $M'/4C \leq t \leq M+2$.

For each $P_i$, let $S_i$ denote the copy of $K_C$ induced by the first $C$ vertices on $P_i$; let $E_i$ denote the copy of $K_C$ induced by the last $C$ vertices on $P_i$; and let $P'_i$ denote the $2C$-path obtain from $P_i$ by deleting all vertices from $S_i$ and $E_i$. (Note that $P'_i$ is certainly non-empty.)
 
\begin{claim}\label{helpful}
Let $W\subseteq V(G_1)$ be arbitrary so that $|W|\leq \eps n'$. Given any $1 \leq i \leq t-1$, there is an $r$-path $P$ in $G$ so that:
\begin{itemize} 
\item[(i)] $V(P)\cap V(G_2) =E_i \cup S_{i+1}$;
\item[(ii)] $|V(P)\cap V(G_1)|=r $;
\item[(iii)] The first $C$ vertices on $P$ are precisely the vertices from $E_i$; 
\item[(iv)] The last $C$ vertices on $P$ are precisely the vertices from $S_{i+1}$;
\item[(v)] $P$ is disjoint from $W$.
\end{itemize}
\end{claim}

\begin{claimproof}
Apply Lemma~\ref{ingredient} with $G,V',n',\eta _3,\sqrt{\rho},d,E_i,S_{i+1},W,r$ playing the roles of $G,U,n',\eta,\rho,d,X,$ $Y,W,r$ to obtain a copy $K$ of $K_r$ in $G_1=G[V']$ such that $V(K) \cap W = \emptyset$ (recall that $E_i \cup S_{i+1}$ is disjoint from $V'$), and there exist $E'_i \subseteq E_i$ and $S'_{i+1} \subseteq S_{i+1}$ such that $|E'_i|=|S'_{i+1}|=r$ and $E'_i \cup S'_{i+1} \subseteq N_G(K)$.

Altogether this implies that $G_1$ contains the desired $r$-path $P$. Indeed, we construct $P$ so that the first $C-r$ vertices on $P$ are those vertices in $E_i\setminus E'_i$ (in an arbitrary order); the next $r$ vertices are the elements from $E'_i$; after that we take the vertices from $K$; then from $S'_{i+1}$; the final $C-r$ vertices on $P$ are from $S_{i+1}\setminus S'_{i+1}$.
\end{claimproof}

\medskip
\noindent
With Claim~\ref{helpful} to hand it is now easy to complete the proof of the theorem.
Suppose for some $ j < t-1$ we have defined vertex-disjoint $r$-paths $P^*_1,\dots, P^* _j$ such that, for each $i\leq j$, $P=P^*_i$ satisfies (i)--(iv) in Claim~\ref{helpful}.
Then define $W$ to be all those vertices in an $r$-path $P^*_1,\dots, P^* _j$ that lie in $G_1$. So $|W|=jr\leq (M+2)r \leq \eps n'$. Claim~\ref{helpful} then implies there is an $r$-path $P^*_{j+1}$ in $G$ that satisfies the conclusion of Claim~\ref{helpful} (where
$j+1$ plays the role of $i$ and $P^*_{j+1}$ the role of $P$).

Thus, we obtain vertex-disjoint $r$-paths $P^*_1,\dots, P^* _t$ such that, for each $i\leq t$, $P=P^*_i$ satisfies (i)--(iv) in Claim~\ref{helpful}.
Consider the concatenation 
$$P^*:= S_1P'_1P^*_1P'_2P^*_2\dots P'_{t-1} P^*_{t-1}P'_tE_t.$$
This induces an $r$-path in $G$ (with many additional edges).
Further, note that by the definition of $P_1$ (and thus $S_1$), the first $C$ vertices on $P^*$ lie in $K^E_{2C+1}$, and so are adjacent in $G$ to every vertex in $E$.
Similarly, the last $C$ vertices on $P^*$ lie in $K^S_{2C+1}$, and so are adjacent in $G$ to every vertex in $S$.
Thus, if we concatenate $P^*$ together with $P_{abs}$ we obtain an $r$-cycle $C^*$ in $G$ (with many additional edges).

Note that, by~(\ref{cover}), $C^*$ covers every vertex in $G$ except for at most $2\eta _3 n$ vertices in $G_2$ and at most $n'=\eta _3 n$ vertices in $G_1$. 
Since $3 \eta _3 n< \eta _2 n$, we may use the absorbing property (Claim~\ref{pabs}(ii)) of $P_{abs}$ to obtain the $r$th power of a Hamilton cycle in $G$, as required.\qed


\section{Lemmas for $H$}\label{sec6}
Our rough aim is to find `compatible' partitions of the vertex sets of $G$ and of $H$ that allow us to apply the embedding lemmas (Lemmas~\ref{target} and~\ref{blowlem}) to complete the embedding of $H$ into $G$.
In this section we state and prove the so-called \emph{Lemmas for $H$}, whose input is some information about the structure of $G$, and whose output is a suitable partition of $H$.

\subsection{Partitioning a graph of low bandwidth: the basic lemma for $H$}
At some stage of the proof, $G$ will return some `ideal' part sizes $\lbrace m_{i,j} : (i,j) \in [\ell]\times[2r]\rbrace$, where $\chi(H) \leq r$. We would then like to find a suitable partition of $H$ the parts of which are close to these ideal sizes (equivalently, a mapping $f$ from $V(H)$ into $[\ell]\times[2r]$ whose pre-images have controlled size). This is the purpose of the next lemma. 
It guarantees that $f$ is a graph homomorphism into $Z^{2r}_\ell$ and produces a small set $B$ such that $f$ restricted to $V(H)\setminus B$ is a graph homomorphism into a $K_{2r}$-factor (this is~$(\mathscr{B}3)$).
Further,~$(\mathscr{B}4)$ says that for  the first few vertices of $H$ (with respect to the bandwidth ordering of $H$), we have control of their images.

Before stating and proving Lemma~\ref{lemmaforH1}, we would like to compare it to Lemma~8 in~\cite{bot}, the Lemma for $H$ in the Bandwidth theorem. There, the assumptions on $H$ are the same (in fact slightly weaker), and the graph $Z^{2r}_\ell$ mentioned above is replaced by a given graph $R$ of large minimum degree which contains a spanning subgraph $S$ (very similar to $Z^{r}_\ell$), which in turn contains a $K_r$-factor. Most edges are (and must be) mapped to the $K_r$-factor, which is much sparser than the $K_{2r}$-factor we have at our disposal. This means that the proof of Lemma~8 in~\cite{bot} is much harder to prove than our Lemma~\ref{lemmaforH1}. Despite this, our lemma does not follow from the statement of Lemma~8 in~\cite{bot}, so we prove it here.

\begin{lemma}[Basic Lemma for $H$]\label{lemmaforH1}
Let $n,r,\ell,\Delta \geq 1$ be integers and let $\beta>0$ be such that $0 < 1/n \ll 1/r,1/\ell,1/\Delta,\beta$.
Let $H$ be a graph on $n$ vertices with $\Delta(H) \leq \Delta$ and assume that $H$ has a labelling $x_1,\ldots,x_n$ of bandwidth at most $\beta n$ and $\chi(H) \leq r$.
Furthermore, suppose $\lbrace m_{i,j} : (i,j)\in[\ell]\times[2r]\rbrace$ is such that $\sum_{(i,j) \in [\ell]\times[2r]}m_{i,j}=n$; $m_{i,j} \geq 10\beta n$ for all $(i,j)\in[\ell]\times[2r]$; and $|m_{i,j}-m_{i,j'}| \leq 1$ whenever $i \in [\ell]$ and $j,j'\in[2r]$.
Let $\chi : V(H) \rightarrow [r]$ be a proper colouring of $H$.
Then there exists a mapping $f:V(H)\rightarrow [\ell]\times[2r]$ and a set of special vertices $B \subseteq V(H)$ with the following properties:
\begin{itemize}
\item[$(\mathscr{B}1)$] $B \cap \lbrace x_1,\ldots,x_{\beta n}\rbrace = \emptyset$ and $|B| \leq 2\ell\beta n$;
\item[$(\mathscr{B}2)$] $\left| |f^{-1}(i,j)| - m_{i,j}\right| \leq 10\beta n$ for every $(i,j) \in [\ell]\times[2r]$;
\item[$(\mathscr{B}3)$] for every edge $uv\in E(H)$, writing $f(u)=:(i,j)$ and $f(v)=:(i',j')$, we have $|i-i'|\leq 1$ and $j \neq j'$. If additionally $u,v \notin B$, then $i=i'$;
\item[$(\mathscr{B}4)$] for all $s \leq \beta n$ we have $f(x_s)=(1,\chi(x_s))$.
\end{itemize}
In particular, $f$ yields a homomorphism from $H$ to $Z^{2r}_{\ell}$.
\end{lemma}
Note that the graph $Z^{2r}_\ell$ which appears in Lemma~\ref{lemmaforH1} will be found in the reduced graph $R$ of $G$: since $G$ is locally dense, $R$ is also locally dense (see Lemma~\ref{inherit}) and thus, by Theorem~\ref{rcycle}, we can find a spanning $(4r-1)$-cycle in $R$, which contains $Z^{2r}_\ell$ (see~(\ref{CZ})).

Recall that each vertex in  $R$ corresponds to a unique cluster in $G$.
In the proof of Theorem~\ref{main}, the homomorphism $f$ from $H$ to $Z^{2r}_{\ell}\subseteq R$ will be a guide  as to which cluster in $G$ we should embed a vertex $x$ into for \emph{most} vertices $x\in V(H)$.
That is, roughly speaking, if $f(x)=(i,j)\in V(R)$, we embed $x$ into the cluster in $G$ corresponding to $(i,j)$. 
Note though that $f$ does not `guide' us as to which vertices from $H$ we should embed into the exceptional set $V_0$ of $G$.
So in the proof of Theorem~\ref{main} we in fact apply Lemma~\ref{lemmaforH1} to an almost spanning subgraph of $H$, rather than $H$ itself; the remaining part of $H$ is then embedded into $G$ 
via an additional Lemma for $H$ (Lemma~\ref{lemmaforH} in Section~\ref{special}). In particular, Lemma~\ref{lemmaforH} governs which vertices from $H$ are embedded into $V_0$.
Property $(\mathscr{B}4)$ of the homomorphism $f$ is used to ensure we can `fit' the two Lemmas for $H$ together to complete the embedding of $H$ into $G$.

The idea of the proof of Lemma~\ref{lemmaforH1} is to first obtain a proper $2r$-colouring $\chi '$ of $H$ such that in any initial segment $x_1,\dots, x_t$ of the bandwidth ordering of $H$, every colour is used roughly the same number of times in $\chi '$.
This then allows us to define $f$ in a sequential way. That is, for some $t_1$ we map each $x_j$ in $\{x_1,\dots,x_{t_1} \}$ to $(1,\chi '(x_j))$; then for some $t_2$ we map each $x_j$ in $\{x_{t_1+1},\dots,x_{t_2}\}$ to $(2,\chi '(x_j))$, and so on.

\smallskip

{\noindent \it Proof of Lemma~\ref{lemmaforH1}.}
Let $N := \lceil 1/(2\beta)\rceil$ and
partition the ordered vertices $x_1,\ldots,x_n$ into consecutive intervals $A_1,A_2,\ldots,A_{2N}$ each of length $\beta n$ (except possibly $A_{2N}$ which could be smaller). 
\COMMENT{AT: deleted definition of $A_i^j$ as don't think it is used}
We view each interval as being ordered with the inherited bandwidth ordering.

We will first define a (proper) $2r$-colouring $\chi' : V(H) \rightarrow [2r]$ by iteratively defining colourings $\chi'_i$ for $i \in [N]$ with the following properties:
\begin{itemize}
\item[$\mathscr{P}_1(i)$] $\chi'_i : \bigcup_{2 \leq t \leq 2i}A_t \rightarrow [2r]$ is a proper colouring of $H[\bigcup_{2 \leq t \leq 2i}A_t]$;
\item[$\mathscr{P}_2(i)$] for all odd $2 \leq t \leq 2i$ we have $\chi'_i(A_t) \subseteq [r]$ and for all even $2 \leq t \leq 2i$ we have $\chi'_i(A_t) \subseteq [2r]\setminus[r]$;
\item[$\mathscr{P}_3(i)$] writing $b_i^j(s) := |\lbrace x \in \bigcup_{2 \leq t \leq 2s}A_t : \chi'_i(x)=j\rbrace|$ for all $j \in [2r]$ and $s \in [i]$, we have $|b_i^j(s)-b_i^{j'}(s)| \leq \beta n$ for all $(j,j')\in[[2r]]^2$ and $s \in [i]$.
\end{itemize}
For $\mathscr{P}_3(i)$ recall that $[[2r]]^2:=[r]^2\cup ([2r]\setminus [r])^2$.
Define $\chi'_1 : A_2 \rightarrow [2r]$ by setting $\chi'_1(x) = \chi(x)+r$. Clearly this satisfies $\mathscr{P}_1(1)$--$\mathscr{P}_3(1)$, in particular as $|A_2| \leq \beta n$.
Suppose we have defined $\chi'_i$ for some $i < N$ satisfying $\mathscr{P}_1(i)$--$\mathscr{P}_3(i)$.
By permuting the sets of colours $[r]$ and $[2r]\setminus[r]$, we can obtain a new proper $2r$-colouring $c_1$ of $H[\bigcup_{2 \leq t \leq 2i} A_t]$ satisfying $\mathscr{P}_1(i)$--$\mathscr{P}_3(i)$ and with the additional property that
\begin{equation}\label{c1}
|c_1^{-1}(1)| \geq \ldots \geq |c_1^{-1}(r)|\quad\text{and}\quad |c_1^{-1}(r+1)| \geq \ldots \geq |c_1^{-1}(2r)|.
\end{equation}

Define $k : A_{2i+1} \cup A_{2i+2} \rightarrow [2r]$ by setting
\begin{equation*}\label{kdef99} 
k(x) = \begin{cases} \chi(x) &\mbox{if } x \in A_{2i+1} \\
\chi(x)+r & \mbox{if } x \in A_{2i+2}. \end{cases} 
\end{equation*}
Clearly $k$ is a proper colouring of $H[A_{2i+1}\cup A_{2i+2}]$ since $\chi$ is.
By permuting the sets of colours $[r]$ and $[2r]\setminus[r]$, we can obtain a new proper colouring $c_2$ of $H[A_{2i+1}\cup A_{2i+2}]$ from $k$ such that
\begin{equation}\label{c2}
|c_2^{-1}(1)| \leq \ldots \leq |c_2^{-1}(r)|\quad\text{and}\quad |c_2^{-1}(r+1)| \leq \ldots \leq |c_2^{-1}(2r)|
\end{equation}
(note that the ordering is reversed compared to~(\ref{c1})).
Finally, define $\chi'_{i+1}$ by setting
\begin{equation}\label{chidef}
\chi'_{i+1}(x) = \begin{cases} c_1(x) &\mbox{if } x \in \bigcup_{2 \leq t \leq 2i}A_t \\
c_2(x) & \mbox{if } x \in A_{2i+1}\cup A_{2i}. \end{cases} 
\end{equation}
The fact that $\mathscr{P}_2(i+1)$ holds is clear from $\mathscr{P}_2(i)$ and the definitions of $c_1$, $k$, $c_2$ and $\chi'_{i+1}$.

To see that $\mathscr{P}_1(i+1)$ holds, let $x,y \in \bigcup_{2 \leq t \leq 2i+2}A_t$ where $xy \in E(H)$.
We need to show that $\chi'_{i+1}(x) \neq \chi'_{i+1}(y)$.
Let $2 \leq t,t' \leq 2i+2$ be such that $x \in A_t$ and $y \in A_{t'}$. Then $|t-t'|\leq 1$ since the intervals $A_j$ respect the bandwidth ordering and each one (except perhaps $A_{2N}$) has size $\beta n$. If $|t-t'|=1$, then $\mathscr{P}_2(i+1)$ implies that one of $\chi'_{i+1}(x),\chi'_{i+1}(y)$ lies in $[r]$ and the other in $[2r]\setminus[r]$, as required.
So we may assume that $t=t'$. If $2 \leq t \leq 2i$, then $(\chi'_{i+1}(x),\chi'_{i+1}(y))=(c_1(x),c_1(y))$. But $c_1$ is a proper colouring since it was obtained from $\chi'_i$ by permuting colours, and $\chi'_i$ is a proper colouring by $\mathscr{P}_1(i)$.
Suppose that $t \in \lbrace 2i+1,2i+2\rbrace$. Then similarly $(\chi'_{i+1}(x),\chi'_{i+1}(y))=(c_2(x),c_2(y))$, and $c_2$ is a proper colouring since it was obtained from the proper colouring $k$ by permuting colours.
Thus $\mathscr{P}_1(i+1)$ holds.

For $\mathscr{P}_3(i+1)$, define for $j \in [2r]$ and $s \in [i+1]$
$$
b_{i+1}^j(s) := \left|\left\lbrace x \in \bigcup_{2 \leq t \leq 2s}A_t : \chi'_{i+1}(x)=j\right\rbrace\right|
$$
and let $b_{i+1}^j := b_{i+1}^j(i+1) = |(\chi'_{i+1})^{-1}(j)|$. Then (\ref{chidef}) implies that $b_{i+1}^j=|c_1^{-1}(j)|+|c_2^{-1}(j)|$ for all $j \in [2r]$. 
Now let $(j,j')\in[[2r]]^2$.
Clearly $|b_{i+1}^j(s)-b_{i+1}^{j'}(s)| \leq \beta n$ for all $s \in [i]$ since this is true for $\chi'_i$ and hence $c_1$.
So it remains to show that $|b_{i+1}^j-b_{i+1}^{j'}| \leq \beta n$.
Equations~(\ref{c1}) and~(\ref{c2}) imply that
 the quantities $|c_1^{-1}(j)|-|c_1^{-1}(j')|$ and $|c_2^{-1}(j)|-|c_2^{-1}(j')|$ are never both positive, and never both negative, since $j$ and $j'$ are in different orders. This implies that\COMMENT{please check my changes here}
\begin{align*}
|b_{i+1}^j-b_{i+1}^{j'}| & = \left| |c_1^{-1}(j)|-|c_1^{-1}(j')| + |c_2^{-1}(j)|-|c_2^{-1}(j')|\right| \\ & \leq \max \left \{ \left| |c_1^{-1}(j)|-|c_1^{-1}(j')|\right|,
\left| |c_2^{-1}(j)|-|c_2^{-1}(j')|\right|
 \right \}.
\end{align*}
Note that $\left| |c_2^{-1}(j)|-|c_2^{-1}(j')|\right|\leq \beta n$.
Further, $c_1$ was obtained from $\chi'_i$ by permuting colours in $[r]$ and in $[2r]\setminus[r]$, so there is some $(q,q') \in [[2r]]^2$ for which $c_1^{-1}(j)=(\chi'_i)^{-1}(q)$ and $c_1^{-1}(j')=(\chi'_i)^{-1}(q')$.
Thus $\left| |c_1^{-1}(j)|-|c_1^{-1}(j')|\right| = |b_i^q(i)-b_i^{q'}(i)|$ which is at most $\beta n$ by $\mathscr{P}_3(i)$.
Thus $\mathscr{P}_3(i+1)$ holds.

Therefore we can obtain a colouring $\chi'_N : V(H) \setminus A_1 \rightarrow [2r]$ satisfying $\mathscr{P}_1(N)$--$\mathscr{P}_3(N)$.
Finally, define $\chi' : V(H) \rightarrow [2r]$ by setting
\begin{equation}\label{chiprimedef}
\chi'(x) = \begin{cases} \chi'_N(x) &\mbox{if } x \in V(H)\setminus A_1 \\
\chi(x) & \mbox{if } x \in A_{1}. \end{cases} 
\end{equation}
The following properties hold:
\begin{itemize}
\item[(i)] $\chi' : V(H) \rightarrow [2r]$ is a proper colouring;
\item[(ii)] for all odd $t \in [2N]$ we have $\chi'(A_t) \subseteq [r]$ and for all even $t\in[2N]$ we have $\chi'(A_t) \subseteq [2r]\setminus[r]$;
\item[(iii)] writing $d^j(s) := |\lbrace x \in \bigcup_{t \in [s]}A_t : \chi'(x)=j\rbrace|$\COMMENT{note change in definition}
 for all $j \in [2r]$ and $s \in [2N]$, we have $|d^j(s)-d^{j'}(s)| \leq 2\beta n$ for all $(j,j')\in[[2r]]^2$ and $s \in [2N]$.
\end{itemize}

\smallskip

Let $M_0=n_0 := 0$.
For all $i \in [\ell]$, let $M_i := \sum_{j \in [2r]}m_{i,j}$; and $n_i := \sum_{t \in [i]}M_t$.
(Note that $n_\ell =n$.)
Let $B_i := \lbrace x_{n_{i-1}+1},\ldots, x_{n_i}\rbrace$.
So $B_1,\ldots,B_{\ell}$ is a partition of $V(H)$ which respects the bandwidth ordering, and each interval inherits the bandwidth ordering.
Let
$$
B := \bigcup_{2 \leq i \leq \ell}\lbrace x_{n_{i-1}+1},\ldots, x_{n_{i-1}+\beta n}\rbrace \cup \bigcup_{1 \leq i \leq \ell-1} \lbrace x_{n_{i}-\beta n+1},\ldots, x_{n_i}\rbrace,
$$
and define $f : V(H) \rightarrow [\ell]\times[2r]$ by setting
\begin{equation}\label{fmapdef}
f(x) := (i,\chi'(x)) \text{ if } x \in B_i. 
\end{equation}
We claim that $f$ is the required mapping.
Note $|B| = 2(\ell-1)\beta n$, 
and if $t \leq n_1-\beta n$, then $x_t \notin B$.
But $n_1-\beta n \geq 9\beta n$, so certainly $\lbrace x_1,\ldots,x_{\beta n}\rbrace \cap B = \emptyset$.
Hence, $(\mathscr{B}1)$ holds.
To show~$(\mathscr{B}2)$, fix $i \in [\ell]$. Choose the smallest $p^- \in [2N]$ such that the first element of $A_{p-}$ lies in $B_i$, and the largest $p^+ \in [2N]$ such that the last element of $A_{p^+}$ lies in $B_i$.
So $B_i$ is the union of $\bigcup_{p^- \leq t \leq p^+}A_t$ together with a proper subset of $A_{p^--1}$ and a proper subset of $A_{p^++1}$.
Thus,
\begin{equation}\label{f1eq}
|f^{-1}(i,j)| = |\lbrace x \in B_i : \chi'(x)=j\rbrace| = d^j(p^+) - d^j(p^--1) \pm (|A_{p^--1}|+|A_{p^++1}|)
\end{equation}
for all $j \in [2r]$.
Let $(j,j')\in[[2r]]^2$. Then the sizes of $f^{-1}(i,j)$ and $f^{-1}(i,j')$ do not differ much:
\begin{equation}\label{fdiff}
|f^{-1}(i,j)-f^{-1}(i,j')| \stackrel{(\ref{f1eq})}{\leq} \left| d^j(p^+)  - d^{j'}(p^+) \right| + \left|  d^j(p^--1) - d^{j'}(p^--1)\right| + 4\beta n \stackrel{(iii)}{\leq} 8\beta n.
\end{equation}
For any fixed $i \in [\ell]$,
$$
S_1 := \sum_{j \in [r]}|f^{-1}(i,j)| \stackrel{(ii),(\ref{fmapdef})}{=} \sum_{t \text{ odd}}|B_i \cap A_t|\quad\text{and}\quad S_2 := \sum_{j \in [2r]\setminus[r]}|f^{-1}(i,j)|=\sum_{t \text{ even}}|B_i \cap A_t|.
$$ 
Therefore $S_1+S_2=|B_i|=M_i$ and $\left|S_1-S_2\right| \leq \beta n$. So $S_1,S_2 = M_i/2 \pm \beta n$. By definition of the $m_{i,j}$ and $M_i$ we have that $|m_{i,j}-M_i/(2r)| \leq 1$ for all $j \in [2r]$.
Now let $j \in [r]$. We have
\begin{align*}
\left| |f^{-1}(i,j)|-m_{i,j}\right| &\leq \left| |f^{-1}(i,j)|-M_i/(2r)\right| + 1 \leq \frac{1}{r}\left| r|f^{-1}(i,j)| - S_1 \right|  + 2\beta n\\
&\leq \frac{1}{r}\sum_{j' \in [r]}\left| |f^{-1}(i,j)|-|f^{-1}(i,j')|\right| + 2\beta n \stackrel{(\ref{fdiff})}{\leq} 10\beta n,
\end{align*}
as required. The case when $j \in [2r]\setminus[r]$ is almost identical.
Thus~$(\mathscr{B}2)$ holds.

Now let $uv \in E(H)$ and write $f(u)=:(i,j)$ and $f(v)=:(i',j')$ for $i,i'\in[\ell]$ and $j,j'\in[2r]$.
Since $|B_t| > \beta n$ for all $t \in [\ell]$ and $u \in B_i$ and $v \in B_{i'}$, we have that $|i-i'| \leq 1$ by consideration of the bandwidth ordering.
We also have $j=\chi'(u)$ and $j'=\chi'(v)$, and $\chi'$ is a proper colouring of $H$, so $j \neq j'$.
Suppose additionally that $u,v \notin B$.
If $i \neq i'$, then $u$ and $v$ are separated by at least $2\beta n$ in the bandwidth ordering, so $uv \not\in E(H)$, a contradiction.

Finally, if $s \leq \beta n$, then $x_s \in A_1 \cap B_1$. So $f(x_s)=(1,\chi'(x))=(1,\chi(x))$ by~(\ref{chiprimedef}). So~$(\mathscr{B}4)$ holds.
\qed

\subsection{Covering exceptional vertices: the second lemma for $H$}\label{special}

The second lemma for $H$ will be used to find an embedding of a short initial segment of $H$ (in bandwidth ordering) into $G$ such that the exceptional set $V_0$, obtained after applying the Regularity lemma, lies in the image of this embedding.
In fact the pre-image of $V_0$ will be a $2$-independent set, which exists because $H$ has small maximum degree and bandwidth.
As well as embedding this initial segment, we would like to find target sets for its neighbours so that eventually we can extend this embedding to the whole of $H$.

\begin{lemma}[Special Lemma for $H$]\label{lemmaforH}
Let $n,r,L \geq 1$ be integers and let 
$0 < 1/n \ll \beta \ll 1/L \ll \eps  \ll \rho \ll \eta \ll d, 1/r,1/\Delta$.\COMMENT{AT: removed $\delta$ from hierarchy -- don't think it is used. Added $1/\Delta$}
\COMMENT{AT: Important for Lemma 15 that it is $\eta \ll d,1/r$ here}
Let $G$ be an $n$-vertex graph, $R$ an $L$-vertex graph and $\lbrace b_1,\ldots,b_r\rbrace \subseteq V(R)$ be such that
\begin{itemize}
\item[$(\mathscr{G}1)$] $G$ has vertex partition $\lbrace V_0\rbrace  \cup \lbrace V_a : a \in V(R)\rbrace$ where $|V_0| \leq \eps n$ and $|V_a|=:m$ for all $a \in V(R)$;
\item[$(\mathscr{G}2)$] each $v \in V_0$ is equipped with a subset $N_v \subseteq V(R)$ with $|N_v| \geq \eta L$
\COMMENT{AT: changed notation as $N(v)$ could get confused with $N_G (v)$}
\item[$(\mathscr{G}3)$] $R$ is $(\rho,d)$-dense and $\delta(R) \geq (1/2+\eta)L$;
\item[$(\mathscr{G}4)$] $R[\lbrace b_1,\ldots,b_r\rbrace] \cong K_r$ and $\lbrace b_1,\ldots,b_r\rbrace$ lies in a copy of $K_{18 r/\eta ^2}$ in $R$.\COMMENT{AT: Important for Lemma 15 that it is $\eta ^2$ here}
\end{itemize}
Then there exists an integer $s \leq \eps^{1/4}n$ such that the following holds.
Let $H$ be a graph on $s+\beta n$ vertices with $\Delta(H) \leq \Delta$ and assume that $H$ has a labelling $x_1,\ldots,x_{s+\beta n}$ of bandwidth at most $\beta n$ and $\chi(H) \leq r$.
Let $X := \lbrace x_1,\ldots,x_{s}\rbrace$ and $Y := \lbrace x_{s+1},\ldots,x_{s+\beta n}\rbrace$.
Let $\chi : V(H) \rightarrow [r]$ be a proper colouring of $H$.
Then there exists a mapping $f:V(H)\rightarrow V(R) \cup V_0$ with the following properties:
\begin{itemize}
\item[$(\mathscr{D}1)$] setting $I := f^{-1}(V_0)$, we have that $I$ is a subset of $X$ which is $2$-independent in $H$, and each  vertex in $V_0$ is mapped onto from a unique vertex in $H$ (so $|I|=|V_0|$);
\COMMENT{AT: with what was written before, we could have had many things mapped to a vertex in $V_0$} 
\item[$(\mathscr{D}2)$] for all $v \in V_0$, setting $W_v := N_H(f^{-1}(v))$, we have $W_v \subseteq X$ and $f(W_v) \subseteq N_v$;
\item[$(\mathscr{D}3)$] $|f^{-1}(a)| \leq \eps^{1/4}m$ for every $a \in V(R)$;
\item[$(\mathscr{D}4)$] for every edge $uv\in E(H)$ such that $f(u),f(v) \notin V_0$, we have $f(u)f(v) \in E(R)$;
\item[$(\mathscr{D}5)$] for all $y \in Y$ we have $f(y)=b_{\chi(y)}$.
\end{itemize}
\end{lemma}

To prove Lemma~\ref{lemmaforH}, we will need an auxiliary result, Lemma~\ref{findF}, which produces a `framework' $F$ in the reduced graph which we will later use to find $f$. This framework $F$ is a $2r$-trail such that for every $v \in V_0$ there is a copy 
$T$ of $K_{2r}$ in $F$ such that $V(T)\subseteq N_v$.

\begin{lemma}\label{findF}
Let $0 < 1/n \ll 1/L \ll \eps  \ll \rho \ll \eta \ll d, 1/r \leq 1$.
Let $G$ be an $n$-vertex graph, $R$ an $L$-vertex graph and $\lbrace b_1,\ldots,b_r\rbrace \subseteq V(R)$ be such that
\begin{itemize}
\item[$(\mathscr{G}1)$] $G$ has vertex partition $\lbrace V_0\rbrace  \cup \lbrace V_a : a \in V(R)\rbrace$ where $|V_0| \leq \eps n$ and $|V_a|=:m$ for all $a \in V(R)$;
\item[$(\mathscr{G}2)$] each $v \in V_0$ is equipped with a subset $N_v \subseteq V(R)$ with $|N_v| \geq \eta L$;
\item[$(\mathscr{G}3)$] $R$ is $(\rho,d)$-dense and $\delta(R) \geq (1/2+\eta)L$;
\item[$(\mathscr{G}4)$] $R[\lbrace b_1,\ldots,b_r\rbrace] \cong K_r$ and $\lbrace b_1,\ldots,b_r\rbrace$ lies in a copy of $K_{18 r/\eta ^2}$ in $R$.
\end{itemize}
\COMMENT{deleted a sentence that I thought wasn't used- please check}
Then there exists an integer $K \leq L^{2r}$ and a subgraph $F \subseteq R$ such that
\begin{itemize}
\item[$(\mathscr{F}1)$] $F$ is a $2r$-trail with ordering $a_1,\ldots,a_{t}$ where $t=(8K+1)r$;
\item[$(\mathscr{F}2)$] there is a partition $V_0=V_0^1\cup \ldots \cup V_0^K$ such that $N_v \supseteq \lbrace a_{8(i-1)r+1},\ldots,a_{8(i-1)r+2r}\rbrace$ for all $v \in V_0^i$ and $|V_0^i| \leq \sqrt{\eps}m/L^{2r-1}$ for all $i \in [K]$;
\item[$(\mathscr{F}3)$] $(a_{t-r+1},\ldots,a_{t})=(b_1,\ldots,b_r)$;
\item[$(\mathscr{F}4)$] every $a \in V(R)$ appears at most $L^{2r-1}/\eps^{1/12}$ times in the sequence $a_1,\ldots,a_{t}$.
\end{itemize}
\end{lemma}

\begin{proof}
Assume without loss of generality that $V(R)=[L]$.
We first prove the following claim.

\begin{claim}\label{findKrs}
There is a $K \leq L^{2r}$ and a set $\mathcal{T} = \lbrace T_1,\ldots,T_K\rbrace$ of $((d/2)^{2r} \eta L)$-extendable copies of $K_{2r}$ in $R$ such that there is a partition $V_0 = V_0^1 \cup \ldots \cup V_0^{K}$ with the property that, for all $k \in [K]$, we have $|V_0^k| \leq \sqrt{\eps}m/L^{2r-1}$ and $T_k \subseteq R[N_v]$ for all $v \in V_0^k$.
\end{claim}

\begin{claimproof}
By Lemma~\ref{nbrhd}(i), we see that $R_v := R[N_v]$ is $(\rho L^2/|N_v|^2,d)$-dense, and hence $(\sqrt{\rho},d)$-dense, where we used $(\mathscr{G}2)$ and the fact that $\rho/\eta^2 < \sqrt{\rho}$.
Lemma~\ref{nbrhd}(iv) implies that $R_v$ contains at least $(d/2)^{\binom{2r+1}{2}} \eta ^{2r} L^{2r}/(2r)!$ copies of $K_{2r}$, each of which is $((d/2)^{2r} \eta L)$-extendable in $R_v$ (and thus $R$).
\COMMENT{We do need these $\eta$ terms here in the 2 parts of this sentence.... so lots of changes to constants cos of this}

Let $T_1,\ldots,T_K$ be the set of $((d/2)^{2r} \eta L)$-extendable copies of $K_{2r}$ in $R$.
So
\begin{equation}\label{Keq99}
K \leq \binom{L}{2r} \leq L^{2r}.
\end{equation}
Then there is a partition $V_0^1 \cup \ldots \cup V_0^{K}$ of $V_0$ into subsets (some of which may be empty) such that for all $k \in [K]$ and $v \in V_0^k$ we have that $R_v \supseteq T_k$ and
$$
|V_0^k| \leq \frac{|V_0|}{(d/2)^{\binom{2r+1}{2}}\eta ^{2r} L^{2r}/(2r)!} \stackrel{(\mathscr{G}1)}{\leq} \frac{\eps m}{(d/2)^{\binom{2r+1}{2}} \eta ^{2r} (1-\eps)L^{2r-1}/(2r)!} \leq \frac{\sqrt{\eps}m}{L^{2r-1}},
$$
as desired.
\end{claimproof}

\medskip
\noindent
Let $\mathcal{T} := \lbrace T_i : i \in [K]\rbrace$ be obtained from the claim.
To complete the proof, we will use the Connecting lemma (Lemma~\ref{connect}) to join the $K_{2r}$s in $\mathcal{T}$ into a $2r$-trail.
In so doing, we have to be careful not to visit any $a \in [L]$ too many times so as to ensure  $(\mathscr{F}4)$ holds.

Suppose, for some $0 \leq i < K-1$ and all $j \in [i]$ we have obtained a copy $P_j=x_j^1\ldots x_j^{6r}$ of $P^{2r}_{6r} \subseteq R$ such that
\begin{itemize}
\item[$\mathscr{P}_1(i)$] $V(T_j)x_j^1\ldots x_j^{6r}$ induces a copy $P_j'$ of $P^{2r}_{8r}$;
\item[$\mathscr{P}_2(i)$] $x_j^1\ldots x_j^{6r}V(T_{j+1})$ induces a copy $P_j''$ of $P^{2r}_{8r}$;
\item[$\mathscr{P}_3(i)$] each $a \in [L]$ lies in at most $\eps^{-1/12}L^{2r-1}/2$ of the $2r$-paths $P_1,\ldots,P_i$.
\end{itemize}
We would like to find $P_{i+1}$ such that $\mathscr{P}_1(i+1)$--$\mathscr{P}_3(i+1)$ hold.
We will say that $a \in [L]$ is \emph{bad} if it appears in at least $\eps^{-1/12}L^{2r-1}/3$ of $P_1,\ldots,P_i$.
Let $D$ be the set of bad $a$. Since each $P_j$ contains $6r$ vertices, we have
$$
|D| {\leq} \frac{6ir}{\eps^{-1/12}L^{2r-1}/3} \leq \frac{18\eps^{1/12} Kr}{L^{2r-1}} \stackrel{(\ref{Keq99})}{\leq} 18\eps^{1/12}rL.
$$
Recall from Claim~\ref{findKrs} that $T_{i+1}$ and $T_{i+2}$ are both $((d/2)^{2r} \eta L)$-extendable copies of $K_{2r}$ in $R$.
Since $\eta \ll d,1/r$, they are $\eta ^2 L$-extendable copies.
Apply Lemma~\ref{connect} with $R,V(T_{i+1}),V(T_{i+2}),D,2r,\eta ^2$ playing the roles of $G,X,Y,W,r,\eta$ to obtain a copy $P_{i+1}$ of $P^{2r}_{6r} = x_{i+1}^1\ldots x_{i+1}^{6r}$ which avoids $D$ and such that $V(T_{i+1})x_{i+1}^1\ldots x_{i+1}^{6r}$ induces a copy $P_{i+1}'$ of $P^{2r}_{8r}$, and $x_{i+1}^1\ldots x_{i+1}^{6r}V(T_{i+2})$ induces a copy $P_{i+1}''$ of $P^{2r}_{8r}$.
So $\mathscr{P}_1(i+1)$ and $\mathscr{P}_2(i+1)$ hold.
Now let $a \in [L]$. If $a \notin V(P_{i+1})$, then $a$ lies in at most $\eps^{-1/12}L^{2r-1}/2$ of $P_1,\ldots,P_{i+1}$ by~$\mathscr{P}_3(i)$.
Otherwise, since $P_{i+1}$ avoids $D$, $a$ lies in at most $\eps^{-1/12}L^{2r-1}/3+1 < \eps^{-1/12}L^{2r-1}/2$ of $P_1,\ldots,P_{i+1}$. So~$\mathscr{P}_3(i+1)$ holds.
Therefore we can find $P_1,\ldots,P_{K-1}$ satisfying $\mathscr{P}_1(K-1)$--$\mathscr{P}_3(K-1)$.

Next we want to find a $2r$-path between $T_K$ and $\lbrace b_1,\ldots,b_r\rbrace$.
Let $\{b'_1,\dots, b'_r\}$ be such that $\lbrace b_1,\ldots,b_r,b'_1,\dots, b'_r\rbrace$ lies in a copy of $K_{18 r/\eta ^2}$ in $R$ (such vertices exist by $(\mathscr{G}4)$).
Apply Lemma~\ref{connect} with $R,V(T_K),$ $\lbrace b_1,\ldots,b_r, b'_1, \dots, b'_r\rbrace,\emptyset,2r,\eta^2$ playing the roles of $G,X,Y,W,r,\eta$ to obtain a copy $P_K$ of $P^{2r}_{6r} = x_{K}^1\ldots x_{K}^{6r}$ such that $V(T_{K})x_{K}^1\ldots x_{K}^{6r}$ induces a copy $P_{K}'$ of $P^{2r}_{8r}$, and furthermore\\
$x_{K}^1\ldots x_{K}^{6r}b_1\ldots b_r b'_1\dots b'_r$ induces a copy  of $P^{2r}_{8r}$; thus $x_{K}^1\ldots x_{K}^{6r}b_1\ldots b_r$ induces a copy $P_{K}''$ of $P^{2r}_{7r}$.
(Note that the vertices $b'_1,\dots, b'_r$ were only introduced so that we could apply Lemma~\ref{connect}.)
Clearly $\mathscr{P}_3(K-1)$ implies that each $a \in [K]$ lies in at most $L^{2r-1}\eps^{-1/12}/2+1$ of $P_1,\ldots,P_K$.

Writing $V(T_i) = \lbrace y_i^1,\ldots,y_i^{2r}\rbrace$ for all $i \in [K]$, $R$ contains a $2r$-trail $F' := \bigcup_{i \in [K]}(P_i' \cup P_i'')$
of length $(8K+1)r=t$, with ordering given by
$$
(a_1,\ldots,a_{t}) := (y_1^1,\ldots y_1^{2r}, x_1^1, \ldots x_1^{6r}, y_2^1,\ldots y_2^{2r}, \ldots, y_K^1,\ldots, y_K^{2r}, x_K^1,\ldots, x_K^{6r}, b_1,\ldots,b_r).
$$
By construction~$(\mathscr{F}1)$ and~$(\mathscr{F}3)$ hold. 

We have that $V(T_i) = \lbrace a_{8(i-1)r+1},\ldots,a_{8(i-1)r+2r}\rbrace$ for all $i \in [K]$, which together with Claim~\ref{findKrs} implies that~$(\mathscr{F}2)$ holds.
Now let $a \in [L]$. Then~$\mathscr{P}_3(K-1)$ implies that $a$ plays the role of some $x_i^j$ with $(i,j)\in[K]\times[6r]$ at most $\eps^{-1/12}L^{2r-1}/2+1$ times.
Since each $T_i$ with $i \in [K]$ is a distinct copy of $K_{2r}$ in $R$, we see that $a$ plays the role of some $y_i^j$ with $(i,j)\in[K]\times[2r]$ at most $\binom{L-1}{2r-1} \leq L^{2r-1}$ times.
Clearly $a$ plays the role of at most one of $b_1,\ldots,b_r$.
Thus the number of times $a$ appears in the sequence $a_1,\ldots,a_t$ is at most $\eps^{-1/12}L^{2r-1}/2 + L^{2r-1}+2 \leq \eps^{-1/12}L^{2r-1}$.
So~$(\mathscr{F}4)$ holds.
\end{proof}

Armed with Lemma~\ref{findF}, we can now prove Lemma~\ref{lemmaforH}.
The proof proceeds by splitting $V(H)$ into segments and assigning each one to a copy of $K_r$ in $R$, according to the framework $F$. For example, the first segment of $V(H)$ will be assigned to $\lbrace a_1,\ldots,a_r\rbrace$, and more specifically, those vertices coloured $i$ by $\chi$ will be mapped to $a_i$.
In those special segments assigned to vertex sets of $K_r$s which lie in $N_v$ for $v \in V_0^i$, we choose $|V_0^i|$ special vertices to be the pre-images of vertices in $V_0^i$.
The property~$(\mathscr{F}4)$ of $F$ will ensure that not too many vertices are mapped to the same cluster of $R$.

\begin{proof}[Proof of Lemma~\ref{lemmaforH}]
Let $G$ and $R$ be as in the statement of the lemma.
Without loss of generality we will assume that $V(R)=[L]$.
Apply Lemma~\ref{findF} to obtain $K \leq L^{2r}$ and $F \subseteq R$ such that
\begin{itemize}
\item[$(\mathscr{F}1)$] $F$ is a $2r$-trail with ordering $a_1,\ldots,a_{t}$ where $t=(8K+1)r$;
\item[$(\mathscr{F}2)$] there is a partition $V_0=V_0^1\cup \ldots \cup V_0^K$ such that $N_v \supseteq \lbrace a_{8(i-1)r+1},\ldots,a_{8(i-1)r+2r}\rbrace$ for all $v \in V_0^i$ and $|V_0^i| \leq \sqrt{\eps}m/L^{2r-1}$ for all $i \in [K]$;
\item[$(\mathscr{F}3)$] $(a_{t-r+1},\ldots,a_{t})=(b_1,\ldots,b_r)$;
\item[$(\mathscr{F}4)$] every $a \in [L]$ appears at most $L^{2r-1}/\eps^{1/12}$ times in the sequence $a_1,\ldots,a_{t}$.
\end{itemize}
Let
$$
s := 8K\eps^{1/3}m/L^{2r-1} \leq 8L\eps^{1/3}m \stackrel{(\mathscr{G}1)}{\leq} 8\eps^{1/3}n \leq \eps^{1/4}n.
$$
For all $i \in [K]$, let
\begin{equation}\label{15sizes}
u_i := |V_0^i| \stackrel{(\mathscr{F}2)}{\leq} \sqrt{\eps}m/L^{2r-1}\quad\text{and}\quad b := \eps^{1/3}m/L^{2r-1} > 100\beta mL \stackrel{(\mathscr{G}1)}{>} 99\beta n.
\end{equation}
Let $H,X,Y$ be as in the statement of the lemma.
Define a partition of $X \cup Y = \lbrace x_1,\ldots,x_{s+\beta n}\rbrace$ into $8K+1$ intervals
$$
B_1^1,B_1^2,\ldots,B^{8}_1,B^1_2,B_2^2,\ldots,B_2^{8},B^1_3,\ldots,B^1_K,B_K^2,\ldots,B_K^{8},B^1_{K+1}
$$
where $|B_i^j|=b$ for all $(i,j) \in [K]\times[8]$; $|B^1_{K+1}|=\beta n$; the first $b$ vertices $x_1,\dots,x_b$ in $X\cup Y $ form $B_1^1$, the next $b$ vertices in $X\cup Y$ form  $B_1^2$, and so on.
In particular, $B^1_{K+1}=Y$ and
 each interval comes equipped with the ordering inherited from the bandwidth ordering of $H$.
The first claim identifies a set $I \subseteq X$ which will be the pre-image of $V_0$ in our desired mapping.
Recall that given a graph $J$ and $A \subseteq V(J)$, we say that $A$ is \emph{$2$-independent} if every pair of vertices in $A$ are at distance at least $3$ in $J$. In other words, $A$ is an independent set and additionally the neighbourhoods of different vertices in $A$ are disjoint.
\medskip
\noindent

\begin{claim}\label{claim1H}
For each $i\in[K]$, there exists a $2$-independent set $I_i \subseteq B^1_i$ (with respect to $H$) of size $u_i$ such that $W(i) := \bigcup_{y \in I_i}N_H(y) \subseteq B^1_i$. Further, $I := \bigcup_{i \in [K]}I_i$ is a $2$-independent set in $H$.
\end{claim}

\begin{claimproof}
Obtain $A_i$ from $B^1_i$ by removing the first $2\beta n$ and last $2\beta n$ elements (which is possible by~(\ref{15sizes})).
Suppose we have obtained a $2$-independent set $I^j \subseteq A_i$ of size $0 \leq j < u_i$.
Then for any $y \in A_i$, the set $I^j \cup \lbrace y \rbrace$ is a $2$-independent set in $H$ of size $j+1$ if $y \notin I^j \cup N_H(y') \cup N_H(N_H(y'))$ for any $y' \in I^j$.
The number of excluded $y$ is at most 
\begin{eqnarray*}
|I^j|+\sum_{x \in I^j}d_H(x) + \sum_{x\in I^j}\sum_{z \in N_H(x)}d_H(z) &\leq& |I^j|(1+\Delta+\Delta^2) \leq 2\Delta^2 u_i \stackrel{(\ref{15sizes})}{\leq} 2\Delta^2\sqrt{\eps}m/L^{2r-1}\\
&\stackrel{(\ref{15sizes})}{<}& b - 4\beta n = |A_i|.
\end{eqnarray*}
Therefore we can find a $2$-independent set $I_i := I^{u_i}$ of size $u_i$ in $A_i$.
This together with the bandwidth property and the definition of $A_i$ implies that $W(i) \subseteq \bigcup_{y \in I_i}(N_H(y) \cup N_H(N_H(y))) \subseteq B^1_i$.
Thus there is no edge between $N_H(I_i)$ and $N_H(I_{i'})$ for $i \neq i'$.
So $I = \bigcup_{i \in [K]}I_i$ is a $2$-independent set in $H$, proving the claim.
\end{claimproof}

\medskip
\noindent
Let $\chi : V(H) \rightarrow [r]$ be the given proper colouring of $H$.
A second claim finds a suitable homomorphism $\phi : V(H) \rightarrow V(F)$ on which $f$ will be based.

\begin{claim}\label{claim2H}
For each $(i,j) \in [K]\times[8] \cup \lbrace (K+1,1)\rbrace$,
let
$$
\phi(x) := a_{(8(i-1)+(j-1))r+\chi(x)}\quad\text{if } x \in B_i^j.
$$
Then $\phi : V(H) \rightarrow V(F)$ is a graph homomorphism such that 
$|\phi^{-1}(a)| \leq \eps^{1/4}m$ for all $a \in [L]$.
\end{claim}

\begin{claimproof}
Note first that if $a_k$ is in the image of $\phi$ for some $k\in \mathbb{N}$, then, recalling~$(\mathscr{F}1)$, we have that $k \in [t]$, so $V(F) \supseteq \phi(V(H))$.
Let us check that $\phi$ is a homomorphism.
Let $xy \in E(H)$.
Let $(i,j),(i',j')\in [K]\times[8] \cup \lbrace (K+1,1)\rbrace$ be such that $x \in B_i^j$ and $y \in B_{i'}^{j'}$.
Since $H$ has bandwidth at most $\beta n$ and $|B_i^j|,|B_{i'}^{j'}| > \beta n$, we must have $(i',j')\in\lbrace (i,j-1),(i,j),(i,j+1)\rbrace$, where we let $(i,9) := (i+1,1)$ and $(i,0) := (i-1,8)$.
So, writing
$$
T_i := \lbrace a_{(8(i-1)+(j-1))r+p} : p \in [r]\rbrace\quad\text{and}\quad T_{i'} := \lbrace a_{(8(i'-1)r+(j'-1))+p}: p \in [r]\rbrace
$$ 
we either have $T_i=T_{i'}$, or $T_i$ and $T_{i'}$ are consecutive intervals in $a_1,\ldots,a_t$ each of length $r$.
In both cases we have $\phi(x) \neq \phi(y)$ (in the first case this follows from the fact that $\chi(x) \neq \chi(y)$).
But~$(\mathscr{F}1)$ now implies that $F[T_i \cup T_{i'}]$ is a clique, so since $\phi(x) \in T_i$ and $\phi(y) \in T_{i'}$ are distinct, $\phi(x)\phi(y) \in E(F)$, as required.

For the final assertion, each $a \in V(F)$ appears at most $L^{2r-1}/\eps^{1/12}$ times in the sequence $a_1,\ldots,a_t$ by~$(\mathscr{F}4)$.
So, writing $\theta : V(H) \rightarrow [t]$ where $\phi(x) = a_{\theta(x)}$, we have
$$
|\phi^{-1}(a)| \leq \frac{L^{2r-1}}{\eps^{1/12}} \cdot \max_{k \in [t]}|\theta^{-1}(k)| \leq \frac{L^{2r-1}b}{\eps^{1/12}} \stackrel{(\ref{15sizes})}{=} \eps^{1/4}m,
$$
as desired.
\end{claimproof}

\medskip
\noindent
Now let $H' := H\setminus I$ where $I := \bigcup_{k \in [K]}I_k$ and $W := \bigcup_{k \in [K]}W(k)$ where $W(k)$ is defined in Claim~\ref{claim1H}.
Note also that $W \subseteq V(H')$ since, by Claim~\ref{claim1H}, $I$ is an independent set.

Let $g : I \rightarrow V_0$ be a bijection such that $g(I_i) =V_0^i$ for all $i \in [K]$ (which is clearly possible by Claim~\ref{claim1H} and~(\ref{15sizes})).
Since $I_i$ is a $2$-independent set in $H$, the set of neighbourhoods $N_H(y)$ is pairwise disjoint over all $y \in I_i$.
So for each $w \in W(i)$, there is a unique $y \in I_i$ for which $w \in N_H(y)$. 
Claim~\ref{claim2H} implies that $|\phi^{-1}(a)| \leq \eps^{1/4}m$ for all $a \in [L]$.

We claim that $f : V(H) \rightarrow [L] \cup V_0$ given by
\begin{equation}\label{fdef}
f(x) = \begin{cases} \phi(x) &\mbox{if } x \in V(H)\setminus I \\
g(x) & \mbox{if } x \in I \end{cases} 
\end{equation}
is the required mapping.
Note that $f(V(H)\setminus I) \subseteq [L]$ and $f(I) \subseteq V_0$.
For~$(\mathscr{D}1)$, 
note that, by Claim~\ref{claim1H} and~(\ref{fdef}), $f(I)=g(I)=V_0$, $|I|=|V_0|$ and $I$ is a $2$-independent subset of $X$.
For~$(\mathscr{D}2)$,
let $v \in V_0$ and $W_v := N_H(f^{-1}(v))$.
Let $k \in [K]$ be such that $v \in V_0^k$.
Then $f^{-1}(v)=g^{-1}(v) \in I_k$. So $W_v \subseteq W(i) \subseteq B^1_i \subseteq  X$.
Let $x \in W_v\subseteq B^1_i$.
By Claim~\ref{claim2H} and $(\mathscr{F}2)$ we have that $f(x)=\phi(x)=a_{8(i-1)+\chi(x)} \in N_v$.
This completes the proof of~$(\mathscr{D}2)$.

For~$(\mathscr{D}3)$, let $a \in V(R)$. Then $f^{-1}(a)\subseteq \phi^{-1}(a)$ has size at most $\eps^{1/4}m$ by Claim~\ref{claim2H}.
For~$(\mathscr{D}4)$, let $uv \in E(H)$ be such that $f(u),f(v) \notin V_0$. So $u,v \in V(H)\setminus I$ and $f(u)=\phi(u)$ and $f(v)=\phi(v)$.
By Claim~\ref{claim2H}, $\phi : V(H)\rightarrow V(F)$ is a homomorphism, so $f(u)f(v) \in E(F) \subseteq E(R)$.
Finally, for~$(\mathscr{D}5)$, we have that $Y \cap I = \emptyset$ by Claim~\ref{claim1H}, so for any $y \in Y$ we have $f(y)=\phi(y)=a_{t-r+\chi(y)}=b_{\chi(y)}$ by~$(\mathscr{F}3)$.
\end{proof}

\section{The lemma for $G$: adjusting cluster sizes}\label{sec7}

Recall the definition of $Z^r_\ell$ from Section~\ref{notation} and in particular that it contains a $K_r$-factor.
Our goal in this section is to prove Lemma~\ref{lemmaforG}. Roughly speaking, it supposes that the reduced graph $R$ of $G$ contains a spanning copy of $Z^{2r}_\ell$, its clusters $V_1,\dots,V_L$ are equally sized, and pairs of clusters corresponding to the $K_{2r}$-factor $\ell\cdot K_{2r}$ in $Z^{2r}_\ell$ are superregular.
Then we can adjust $V_1,\ldots,V_L$ slightly by reallocating a small number of vertices so that they have given sizes, at the expense of now having superregular pairs corresponding to a $K_r$-factor $2\ell \cdot K_r$.

To formalise the structural properties we need from $G$, we make the following definition (very similar to Definition~8.1 in~\cite{st1}).


\begin{definition}[$r$-Cycle structure]\label{cyclestructure}
Given integers $n,\ell,r$, a graph $G$ on $n$ vertices, and constants $\eps,\delta>0$, we say that $G$ has an \emph{$(R,\ell,r,\mathcal{V},\eps,\delta)$-cycle structure} $\mathcal{C}$ if the following hold:
\begin{itemize}
\item[$(\mathscr{C}1)$] $\mathcal{V} = \lbrace V_0\rbrace \cup \lbrace V_{i,j}:(i,j) \in [\ell]\times[r]\rbrace$ is a partition of $V(G)$, where $|V_0| \leq \eps n$.
\item[$(\mathscr{C}2)$] $R$ has vertex set $[\ell]\times[r]$ and $R \supseteq Z^{r}_\ell$ and $G[V_{i,j},V_{i',j'}]$ is $(\eps,\delta)$-regular whenever $(i,j)(i',j') \in E(R)$;
\item[$(\mathscr{C}3)$] $G[V_{i,j},V_{i,j'}]$ is $(\eps,\delta)$-superregular whenever $i \in [\ell]$ and $1 \leq j < j' \leq r$.
\end{itemize}
We say that $\mathcal{V}$ \emph{induces} $\mathcal{C}$. If $V_0 = \emptyset$ we say that $\mathcal{C}$ is \emph{spanning}.
\end{definition}

The next definition concerns a convenient relabelling of the vertex set of a graph, which we will use for the reduced graph $R$.

\begin{definition}[Bijection $\phi^{2r}_\ell$]\label{bijdef}
Given integers $r,\ell$, define $\phi^{2r}_{\ell} : [\ell]\times[2r]\rightarrow[2\ell]\times[r]$ by setting
\begin{equation}\label{phi2rldef}
\phi^{2r}_{\ell}(i,j)=\left((2i-1)+\left\lfloor\frac{j}{r}\right\rfloor,j-\left(\left\lceil \frac{j}{r}\right\rceil-1\right) r\right),\quad\text{for all }(i,j)\in[\ell]\times[2r].
\end{equation}
It is easy to check that $\phi^{2r}_{\ell}$ is a bijection and
\begin{alignat*}{6}
\phi^{2r}_{\ell}&(1,1)\ldots\phi^{2r}_{\ell}&&(1,r)\phi^{2r}_{\ell}&&(1,r+1)\ldots\phi^{2r}_{\ell}&&(1,2r)\ldots\phi^{2r}_{\ell}&&(\ell,r+1)\ldots\phi^{2r}_{\ell}&&(\ell,2r)\\
= &(1,1)\ldots&&(1,r)&&(2,1)\ldots&&(2,r)\ldots&&(2\ell,1)\ldots&&(2\ell,r).
\end{alignat*}
This implies that for all $a \in [2\ell]$ and distinct $b,b'\in[r]$, there are $i \in [\ell]$ and $(j,j')\in[[2r]]^2$ such that $(\phi^{2r}_{\ell}(i,j),\phi^{2r}_{\ell}(i,j'))=((a,b),(a,b'))$.

Given a graph $R$ and a bijection $\phi : V(R) \rightarrow V$ to some set $V$, we write $\phi(R)$ for the graph with vertex set $\lbrace \phi(x) : x \in V(R)\rbrace$ and edge set $\lbrace \phi(x)\phi(y) : xy \in E(R)\rbrace$. So $\phi(R) \cong R$.
\end{definition}

In the language of Definition~\ref{cyclestructure}, the main result of this section states that, given a graph with a (spanning) $2r$-cycle structure, we can obtain from it an $r$-cycle structure which is almost balanced, but the exact deviation from perfect balancedness can be controlled.

\begin{lemma}[Lemma for $G$]\label{lemmaforG}
Let $n,\ell,m,r \in \mathbb{N}$ and $0 < 1/n \ll \xi \ll 1/\ell \ll \eps \ll \delta < 1/r$.
Suppose that $G$ is a graph on $n$ vertices with a spanning $(R,\ell,2r,\mathcal{V},\eps,\delta)$-cycle structure, where $\mathcal{V} = \lbrace V_{i,j} : (i,j)\in[\ell]\times[2r]\rbrace$ and $|V_{i,j}| = m$ for all $(i,j)\in[\ell]\times[2r]$.
Let $\lbrace \tau_{a,b} \in \mathbb Z : (a,b)\in[2\ell]\times[r]\rbrace$ be such that $0 \leq \tau_{a,b}\leq \eps m$ for all $(a,b)\in[2\ell]\times[r]$.
Then there exist positive integers $\lbrace m_{a,b} : (a,b)\in[2\ell]\times[r]\rbrace$ such that
\begin{itemize}
\item[$(\mathscr{L}1)$] $\sum_{(a,b)\in[2\ell]\times[r]}(m_{a,b}+\tau_{a,b})=n$ and $m_{a,b} \geq (1-\sqrt{\eps})m$ and $|m_{a,b}-m_{a,b'}| \leq 1$ for all $a \in [2\ell]$ and $b,b' \in [r]$;
\item[$(\mathscr{L}2)$] given any $\lbrace n_{a,b}\in \mathbb N : (a,b)  \in [2\ell]\times[r]\rbrace$ with $\sum_{(a,b)\in[2\ell]\times[r]}(n_{a,b}+\tau_{a,b})=n$ and $|m_{a,b}- n_{a,b}| \leq \xi n$, there is a partition $\mathcal{X} = \lbrace X_{a,b} : (a,b)\in[2\ell]\times[r] \rbrace$ of $V(G)$ with $|X_{a,b}| = n_{a,b}+\tau_{a,b}$ and $|X_{a,b} \bigtriangleup V_{(\phi^{2r}_\ell)^{-1}(a,b)}| \leq \sqrt{\eps}m$ for all $(a,b)\in[2\ell]\times[r]$ such that
$G$ has a spanning $(\phi^{2r}_\ell(R),2\ell,r,\mathcal{X},\eps^{1/3},\delta/2)$-cycle structure.
\end{itemize}
\end{lemma}

\begin{proof}
Note that
\begin{equation}\label{msize}
2r\ell m=n.
\end{equation}
For each $(i,j)\in[\ell]\times[2r]$,
choose $A_{i,j} \subseteq V_{i,j}$ satisfying
\begin{equation}\label{Aij}
|A_{i,j}|=\tau_{\phi^{2r}_\ell(i,j)}
\end{equation}
and let
\begin{equation}\label{Yij}
Y_{i,j} := V_{i,j}\setminus A_{i,j},\quad\text{so}\quad (1-\eps)m \leq |Y_{i,j}| \leq m.
\end{equation}
Let $\mathcal{Y} :=\lbrace Y_0 \rbrace \cup \lbrace Y_{i,j} : (i,j)\in[\ell]\times[2r]\rbrace$ where
$$
Y_0 := V(G)\setminus \bigcup_{(i,j)\in[\ell]\times[2r]}Y_{i,j} = \bigcup_{(i,j)\in[\ell]\times[2r]}A_{i,j}.
$$

Given a vertex $v \in V(G)$ and $(i,j)\in[\ell]\times[2r]$, we will say that $v \rightarrow Y_{i,j}$ is \emph{valid} if
\begin{itemize}
\item $j\in[r]$ and $d_G(v,Y_{i,j'}) \geq (\delta-2\eps)m$ for all $j' \in [r]\setminus\lbrace j\rbrace$; or
\item $j\in[2r]\setminus[r]$ and $d_G(v,Y_{i,j'}) \geq (\delta-2\eps)m$ for all $j' \in ([2r]\setminus [r])\setminus \lbrace j\rbrace$.
\end{itemize}
The first claim furnishes us with many pairs $(v,Y_{i',j'})$ such that $v \in Y_{i,j}$ and $v \rightarrow Y_{i',j'}$ is valid.

\begin{claim}\label{goodvx}
Let $i \in [\ell]$
and suppose that
$1 \leq j \leq r < t \leq 2r$ or $1 \leq t \leq r < j \leq 2r$.
Then every vertex $v \in Y_{i,j}$ is such that $v\rightarrow Y_{i,j}, Y_{i,t}$ is valid, and at least $(1-\sqrt{\eps})m$ are such that $v \rightarrow Y_{i+1,j}, Y_{i+1,t}$ are also valid.
(Here e.g. $Y_{\ell+1,j}:=Y_{1,j}$.)
\end{claim}

\begin{claimproof}
Let $t,j$ be as in the statement. Since, by~$(\mathscr{C}3)$, $G[V_{i,j},V_{i,j'}]$ is $(\eps,\delta)$-superregular for all $j' \in [2r]\setminus \lbrace j \rbrace$, we have that every vertex $v \in V_{i,j}$ has at least $\delta|V_{i,j'}|$ neighbours in $V_{i,j'}$.
Thus every vertex $v \in Y_{i,j} \subseteq V_{i,j}$ has at least $\delta m-\eps m \geq (\delta-2\eps)m$ neighbours in $Y_{i,j'}$.
In particular, $v \rightarrow V_{i,j},V_{i,t}$ is valid.

From the definition of regularity, one can see the following. If $G[A,B]$ is an $(\eps,\delta)$-regular graph, then there are less than $\eps|A|$ vertices with less than $(\delta-\eps)|B|$ neighbours in $B$.
Thus, if $S_{i,j}$ is a subset of $N_{i,j} := \lbrace (i',j')\in V(R) : G[V_{i,j},V_{i',j'}] \text{ is } (\eps,\delta)\text{-regular}\rbrace$, we see that there are at least $(1-\eps|S_{i,j}|)|V_{i,j}|$ vertices in $V_{i,j}$ with at least $(\delta-\eps)m $ neighbours in $V_{i',j'}$ for all $(i',j')\in S_{i,j}$, and hence at least $(\delta-2\eps)m$ neighbours in $Y_{i',j'}$.

Recall that, since $Z^{2r}_\ell \subseteq R$ by~$(\mathscr{C}2)$, we have that $N_{i,j} \supseteq \lbrace (i,j'), (i+1,j') : j' \in [2r]\setminus\lbrace j \rbrace \rbrace$.
Thus the second assertion of the claim follows by taking $S_{i,j} := \lbrace (i+1,j') : j'\in[2r]\setminus\lbrace j,t\rbrace\rbrace$ and using the fact that $(1-|S_{i,j}|\eps)|V_{i,j}|-|A_{i,j}| \geq (1-(2r-2)\eps)m - \eps m \geq (1-\sqrt{\eps})m$. 
\end{claimproof}

\medskip
\noindent
Next we prove the following claim, which will give us a `balanced' partition.

\begin{claim}\label{Uclaim}
$V(G)$ has a partition $\lbrace Y_0 \rbrace \cup \lbrace U_{i,j} : (i,j)\in[\ell]\times[2r]\rbrace$ such that the following hold for all $i \in [\ell]$:
\begin{itemize}
\item[($\mathscr{U}$1)] $||U_{i,j}|-|U_{i,j'}|| \leq 1$ for all $(j,j') \in [[2r]]^2$;
\item[($\mathscr{U}$2)] $|Y_{i,j}\bigtriangleup U_{i,j}| \leq r\eps m$ for all $j\in [2r]$;
\item[($\mathscr{U}$3)] if $j \in [r]$ then $U_{i,j}\setminus Y_{i,j} \subseteq \bigcup_{k \in [2r]\setminus[r]}Y_{i,k}$ and if $k \in [2r]\setminus[r]$ then $U_{i,k}\subseteq Y_{i,k}$.
\end{itemize}
\end{claim}


\begin{claimproof}
Fix an $i \in [\ell]$, and, to simplify notation, let $A_j := Y_{i,j}$, $a_j := |A_j|$, $B_j := Y_{i,r+j}$ and $b_j := |B_j|$ for all $j \in [r]$.
Suppose without loss of generality that $a_1 \geq \ldots \geq a_r$ and $b_1 \geq \ldots \geq b_r$.
Let
\begin{equation}\label{S}
S := \max\left\lbrace \sum_{j \in [r]}(a_1-a_j), \sum_{j\in[r]}(b_j-b_r)\right\rbrace \stackrel{(\ref{Yij})}{\leq} r\eps m.
\end{equation}

Now let $A_j(0):=A_j$ and $B_j(0) := B_j$, and $a_j(0) := |A_j(0)|$ and $b_j(0) := |B_j(0)|$ for all $j \in [r]$.
Do the following for each $0 \leq s < S$.
Fix $t^-,t^+ \in [r]$ such that $a_{t^-}(s) \leq a_j(s)$ and $b_{t^+}(s) \geq b_j(s)$ for all $j \in [r]$. Choose $x \in B_{t^+} \cap B_{t^+}(s)$ and let
$$
A_{j}(s+1) := \begin{cases} A_{j}(s)\cup\lbrace x \rbrace &\mbox{if } j=t^- \\
A_{j}(s) & \mbox{if } j\in[r]\setminus\lbrace t^- \rbrace; \end{cases}
$$
$$
B_{j}(s+1) := \begin{cases} B_{j}(s)\setminus\lbrace x \rbrace &\mbox{if } j=t^+ \\
B_{j}(s) & \mbox{if } j\in[r]\setminus\lbrace t^+ \rbrace. \end{cases}
$$
Let $a_j(s+1):= |A_j(s+1)|$ and $b_j(s+1) := |B_j(s+1)|$ for all $j\in[r]$.
The following properties are clear:
\begin{itemize}
\item[(i)] for all $0 \leq s < S$ and $j \in [r]$ we have $A_j(s) \supseteq A_j$ and $A_j(s)\setminus A_j \subseteq \bigcup_{k\in [r]}B_k$, and $B_j(s) \subseteq B_j$. Furthermore, for all $j \in [r]$ we have $\sum_{j\in[r]}|A_j(s)\setminus A_j| = \sum_{k\in[r]}|B_k\setminus B_k(s)| = s$;
\item[(ii)] letting $s_1 := \sum_{j\in[r]}(a_1-a_j)$, we have that $a_1(s_1)=\ldots = a_r(s_1)=a_1$; and for each $s>s_1$ we have $|a_j(s)-a_{j'}(s)|\leq 1$;
\item[(iii)] letting $s_2 := \sum_{j \in [r]}(b_j-b_r)$, we have that $b_1(s_2)=\ldots = b_r(s_2)=b_r$; and for each $s>s_2$ we have $|b_j(s)-b_{j'}(s)|\leq 1$.
\end{itemize}
Now let $U_{i,j} := A_j(S)$ if $j\in[r]$ and $U_{i,j} := B_{j-r}(S)$ if $j\in[2r]\setminus[r]$.
For ($\mathscr{U}$1), the fact that $S=\max\lbrace s_1,s_2\rbrace$ together with~(ii) and~(iii) implies that $|a_j(S)-a_{j'}(S)| \leq 1$ and $|b_j(S)-b_{j'}(S)| \leq 1$ for all $j,j'\in[r]$. So~$(\mathscr{U}1)$ holds.
For~($\mathscr{U}$2), we have by~(i) that
\begin{align*}
|U_{i,j} \bigtriangleup Y_{i,j}| &= |U_{i,j}\setminus Y_{i,j}| = |A_j(S)\setminus A_j| \leq S \stackrel{(\ref{S})}{\leq} r\eps m\quad&&\text{if }j \in [r],\quad\text{and}\\
|U_{i,j} \bigtriangleup Y_{i,j}| &= |Y_{i,j}\setminus U_{i,j}| = |B_{j-r}\setminus B_{j-r}(S)| \leq S \stackrel{(\ref{S})}{\leq} r\eps m\quad&&\text{if }j \in [2r]\setminus[r].
\end{align*}
Finally, ($\mathscr{U}$3) follows immediately from~(i).
\end{claimproof}

\medskip
\noindent
The next claim shows that we can modify $\lbrace U_{i,j}\rbrace$ further to obtain a new partition with clusters of given sizes (each of which does not differ much from $|U_{i,j}|$).

\begin{claim}\label{finalclaim}
Let $\lbrace Y_0 \rbrace \cup \lbrace U_{i,j}:(i,j)\in[\ell]\times[2r]\rbrace$ be any partition of $V(G)$ satisfying $(\mathscr{U}1)$--$(\mathscr{U}3)$.
Let $\lbrace n'_{i,j} : (i,j) \in [\ell]\times[2r]\rbrace $ be such that $\sum_{(i,j)\in[\ell]\times[2r]}n'_{i,j}= \sum_{(i,j)\in[\ell]\times[2r]}|U_{i,j}| $ and $||U_{i,j}| - n'_{i,j}| \leq \xi n$ for all $(i,j)\in[\ell]\times[2r]$. Then
$V(G)$ has a partition $\lbrace Y_0 \rbrace \cup \lbrace W_{i,j} : (i,j) \in [\ell]\times[2r]\rbrace$ such that the following hold for all $(i,j)\in[\ell]\times[2r]$:
\begin{itemize}
\item[($\mathscr{W}$1)] $|W_{i,j}| = n'_{i,j}$;
\item[($\mathscr{W}$2)] $|W_{i,j}\bigtriangleup U_{i,j}| \leq \eps m$;
\item[($\mathscr{W}$3)] for every $v \in W_{i,j}$ we have that $v \rightarrow Y_{i,j}$ is valid. 
\end{itemize}
\end{claim}

\begin{claimproof}
Let
\begin{equation}\label{Keq}
K := 2r\ell\xi n \stackrel{(\ref{msize})}{=} 4r^2\ell^2\xi m \leq \frac{\eps m}{2}.
\end{equation}
Suppose, for some $0 \leq k < K/2$, we have found for each $(i,j) \in [\ell]\times[2r]$ subsets $U^k_{i,j} \subseteq V(G)$ such that the following hold:
\begin{itemize}
\item[$\mathscr{A}_1(k)$] $\lbrace Y_0\rbrace \cup \lbrace U^k_{i,j} : (i,j)\in[\ell]\times[2r]\rbrace$ is a partition of $V(G)$;
\item[$\mathscr{A}_2(k)$] for all $v \in U^k_{i,j}$ we have that $v \rightarrow Y_{i,j}$ is valid;
\item[$\mathscr{A}_3(k)$] for all $(i,j)\in[\ell]\times[2r]$ we have $|U^k_{i,j}\bigtriangleup U_{i,j}| \leq 2k$;
\item[$\mathscr{A}_4(k)$] $\sum_{(i,j)\in[\ell]\times[2r]}||U^k_{i,j}|-n'_{i,j}| \leq 2(r\ell \xi n-k)$.
\end{itemize}

We claim that we can set $U^0_{i,j}:=U_{i,j}$ for all $(i,j)\in[\ell]\times[2r]$.
Indeed, $\mathscr{A}_1(0)$ holds by Claim~\ref{Uclaim}.
For $\mathscr{A}_2(0)$, let $(i,j)\in[\ell]\times[2r]$ and let $v \in U_{i,j}$. If $v \in Y_{i,j}$, then $v \rightarrow Y_{i,j}$ is valid by~Claim~\ref{goodvx}. Otherwise, $v \in U_{i,j}\setminus Y_{i,j}$.
Note that by ($\mathscr{U}$3) this implies $j \in [r]$ and further
 $v \in \bigcup_{k \in [2r]\setminus[r]}Y_{i,k}$. So $v\rightarrow Y_{i,j}$ is valid by Claim~\ref{goodvx}.
Property $\mathscr{A}_3(0)$ vacuously holds and $\mathscr{A}_4(0)$ holds since $||U_{i,j}|-n'_{i,j}| \leq \xi n$ for all $(i,j)\in[\ell]\times[2r]$.

If $|U^k_{i,j}|=n'_{i,j}$ for all $(i,j)\in[\ell]\times[2r]$, then we stop.
Otherwise, we will obtain sets $U^{k+1}_{i,j}$ from $U^k_{i,j}$.
There must exist $(i^-,j^-),(i^+,j^+) \in [\ell]\times[2r]$ for which $|U^k_{i^-,j^-}| \leq n'_{i^-,j^-}-1$ and $|U^k_{i^+,j^+}| \geq n'_{i^+,j^+}+1$.

We will say that $(i_1,j_1)\rightarrow (i_2,j_2) \rightarrow \ldots \rightarrow (i_s,j_s)$ is a \emph{good chain (of length $s$)} if for all $p \in [s-1]$ there exist at least $(1-\sqrt{\eps})m$ vertices $v \in Y_{i_p,j_p}$ such that $v\rightarrow Y_{i_{p+1},j_{p+1}}$ is valid.
Claim~\ref{goodvx} implies that the following are good chains of length 3 (where here and for the remainder of the proof of Claim~\ref{finalclaim} addition is modulo $\ell$):
\begin{align*}
&(i^+,j^+) \rightarrow (i^+,j^-+r) \rightarrow (i^++1,j^-) &&\text{ if } j^+,j^- \in [r]\\
&(i^+,j^+) \rightarrow (i^+,j^--r) \rightarrow (i^++1,j^-) &&\text{ if } j^+,j^- \in [2r]\setminus[r]\\
&(i^+,j^+) \rightarrow (i^+,j^-) \rightarrow (i^++1,j^-) &&\text{ otherwise,}
\end{align*}
and further, in all cases and for all $t \geq 0$, the chain $(i^++t,j^-)\rightarrow (i^++t+1,j^-)$ of length 2 is good.
Together this implies that in all cases there is a good chain
$$
(i^+,j^+) =: (i_1,j_1)\rightarrow \ldots \rightarrow (i_{S},j_{S}) := (i^-,j^-)
$$
of some length $S$, where we choose the shortest such chain. As a crude estimate, we have, say, $S \leq 2\ell$, and $(i_s,j_s)\neq(i_{s'},j_{s'})$ for any distinct $s,s' \in [S]$ (or we could find a shorter chain).

We will exchange vertices between successive clusters according to this chain.
For each $s \in [S]$, there are by definition at least $(1-\sqrt{\eps}) m$ vertices $v \in Y_{i_s,j_s}$ such that $v \rightarrow Y_{i_{s+1},j_{s+1}}$ is valid.
The number of these vertices which additionally lie in $U^k_{i_s,j_s}$ is by $(\mathscr{U}2)$,~$\mathscr{A}_3(0)$ and~(\ref{Keq}) at least
$
(1-\sqrt{\eps})m - 2k-r\eps m  >  m/2$.
So we can find $x_s \in U^k_{i_s,j_s}$ such that $x_s \rightarrow Y_{i_{s+1},j_{s+1}}$ is valid.
For each $(i,j)\in[\ell]\times[2r]$, set
$$
U^{k+1}_{i,j} = \begin{cases} U^k_{i,j}\setminus \lbrace x_1\rbrace &\mbox{if } (i,j)=(i_1,j_1) \\
U^k_{i,j} \cup \lbrace x_{s-1}\rbrace \setminus \lbrace x_s\rbrace & \mbox{if } (i,j)=(i_s,j_s) \text{ for some }2 \leq s < S\\
U^k_{i,j} \cup \lbrace x_{S-1}\rbrace & \mbox{if } (i,j)=(i_S,j_S)\\
U^k_{i,j} & \mbox{otherwise.}\end{cases} 
$$
Property~$\mathscr{A}_1(k+1)$ holds by~$\mathscr{A}_1(k)$, the definition of $U^{k+1}_{i,j}$ and the fact that each pair in the chain is distinct.
Property~$\mathscr{A}_2(k)$ and the choice of $x_s$ imply that $\mathscr{A}_2(k+1)$ holds.
We have
$$
|U^{k+1}_{i,j}\bigtriangleup Y_{i,j}| \leq |U^{k+1}_{i,j}\bigtriangleup U^k_{i,j}| + |U^k_{i,j}\bigtriangleup Y_{i,j}| \stackrel{\mathscr{A}_3(k)}{\leq} 2(k+1),
$$
proving~$\mathscr{A}_3(k+1)$ (note here we are again using the fact that each pair in our chain is distinct).
Finally, observe that $||U^{k+1}_{i^\pm,j^\pm}|-n'_{i^\pm,j^\pm}| = ||U^{k}_{i^\pm,j^\pm}|-n'_{i^\pm,j^\pm}| -1$ and $|U^{k+1}_{i,j}|=|U^k_{i,j}|$ for all other pairs $(i,j)$.
Therefore
$$
\sum_{(i,j)\in[\ell]\times[2r]} ||U^{k+1}_{i,j}|-n'_{i,j}| = \sum_{(i,j)\in[\ell]\times[2r]}||U^k_{i,j}|-n'_{i,j}| - 2 \stackrel{\mathscr{A}_4(k)}{\leq} 2(r\ell\xi n - (k + 1)),
$$
proving~$\mathscr{A}_4(k+1)$.
So, for each $0 \leq k \leq K/2$, either the procedure has terminated, or we are able to proceed to step $k+1$.
Therefore there is some $p \leq K/2$ such that $\sum_{(i,j)\in[\ell]\times[2r]}||U^p_{i,j}|-n'_{i,j}|=0$.
Note that, by $\mathscr{A}_3(p)$, we have
$$
|U^p_{i,j} \bigtriangleup U_{i,j}| \leq 2p \leq K \stackrel{(\ref{Keq})}{\leq} \eps m
$$
for all $(i,j)\in[\ell]\times[2r]$.
Thus setting $W_{i,j} := U^p_{i,j}$ yields the required partition.
\end{claimproof}

\medskip
\noindent
Apply Claim~\ref{Uclaim} to obtain $\lbrace U_{i,j} : (i,j)\in[\ell]\times[2r]\rbrace$ satisfying~$(\mathscr{U}1)$--$(\mathscr{U}3)$.

Let $\phi := \phi^{2r}_\ell$ as in Definition~\ref{bijdef}. 
Let
$$
U'_{a,b} := U_{\phi^{-1}(a,b)}\quad\text{and}\quad m_{a,b} := |U'_{a,b}|\quad\text{for all}\quad(a,b) \in [2\ell]\times[r].
$$
We claim that $\lbrace m_{a,b}\rbrace$ satisfies~$(\mathscr{L}1)$.
Indeed,~$(\mathscr{U}1)$ implies that $|m_{a,b}-m_{a,b'}|\leq 1$ for all $a \in [2\ell]$ and $b,b'\in[r]$, and further, writing $\phi^{-1}(a,b) =: (i,j)$,
$$
m_{a,b} = |U_{i,j}| \stackrel{(\mathscr{U}2)}{\geq} |Y_{i,j}|-r\eps m \stackrel{(\ref{Aij})}{\geq} |V_{i,j}| - (r+1)\eps m = (1-(r+1)\eps)m \geq (1-\sqrt{\eps})m.
$$
Finally,
$$
\sum_{(a,b)\in[2\ell]\times[r]} m_{a,b} = \sum_{(i,j)\in[\ell]\times[2r]}|U_{i,j}| = n-|Y_0| = n-\sum_{(a,b)\in[2\ell]\times[r]}\tau_{a,b},
$$
so~$(\mathscr{L}1)$ holds.

Now let $\lbrace n_{a,b} \in \mathbb N: (a,b)\in[2\ell]\times[r]\rbrace$ satisfy $\sum_{(a,b)\in[2\ell]\times[r]}(n_{a,b}+\tau_{a,b}) =n$ and $|m_{a,b}- n_{a,b}| \leq\xi n$.
Let
\begin{equation}\label{neq}
n_{i,j}' := n_{\phi(i,j)}\quad\text{for all }(i,j)\in[\ell]\times[2r].
\end{equation}
Apply Claim~\ref{finalclaim} with input partition $\lbrace Y_0\rbrace \cup \lbrace U_{i,j} : (i,j)\in[\ell]\times[2r]\rbrace$ and input sizes $\lbrace n'_{i,j} : (i,j)\in[\ell]\times[2r] \rbrace$ to obtain a partition $\lbrace Y_0 \rbrace \cup \lbrace W_{i,j} : (i,j)\in[\ell]\times[2r]\rbrace$ satisfying $(\mathscr{W}1)$--$(\mathscr{W}3)$. Let
\begin{equation}\label{Weq}
X_{a,b} := W_{\phi^{-1}(a,b)} \cup A_{\phi^{-1}(a,b)} \quad\text{for all }(a,b)\in[2\ell]\times[r].
\end{equation}
We claim that $\mathcal{X} := \lbrace X_{a,b} : (a,b)\in[2\ell]\times[r]\rbrace$ is the required partition for~$(\mathscr{L}2)$.
For all $(a,b)\in[2\ell]\times[r]$ we have
$$
|X_{a,b}| \stackrel{(\ref{Weq})}{=} |W_{\phi^{-1}(a,b)}|+|A_{\phi^{-1}(a,b)}| \stackrel{(\ref{Aij}),(\mathscr{W}1)}{=} n'_{\phi^{-1}(a,b)}+\tau_{a,b} \stackrel{(\ref{neq})}{=} n_{a,b}+\tau_{a,b},
$$
as required. Also, writing $(i,j):= \phi^{-1}(a,b) \in [\ell]\times[2r]$ and recalling that $A_{i,j}\subseteq V_{i,j}$, we have
\begin{eqnarray}
\label{Xtriangle}|X_{a,b} \bigtriangleup V_{\phi^{-1}(a,b)}| &\stackrel{(\ref{Weq})}{=}& |W_{i,j}\bigtriangleup V_{i,j}|
\leq |W_{i,j}\bigtriangleup U_{i,j}| + |U_{i,j}\bigtriangleup V_{i,j}|\\
\nonumber &\stackrel{(\mathscr{U}2),(\mathscr{W}2)}{\leq}& 2 r\eps m \leq 3r\eps|X_{a,b}| \leq \sqrt{\eps}m.
\end{eqnarray}

Lastly, we need to check that $\mathcal{X}$
induces a $(\phi(R),2\ell,r,\mathcal{X},\eps^{1/3},\delta/2)$-cycle structure.
That is, we need to check that $(\mathscr{C}1)$--$(\mathscr{C}3)$ hold.
Property~$(\mathscr{W}1)$ implies that $\mathcal{X} = \lbrace X_{a,b} : (a,b)\in[2\ell]\times[r]\rbrace = \lbrace W_{i,j}\cup A_{i,j} : (i,j)\in[\ell]\times[2r]\rbrace$ is a partition of $V(G)$.
So $(\mathscr{C}1)$ holds.
Now, by~(\ref{CZ}) and Definition~\ref{bijdef}, we see that $\phi(R)$ has vertex set $[2\ell]\times[r]$ and, since $Z^{2r}_\ell \subseteq R$, we have $Z^r_{2\ell} \subseteq \phi(R)$ (with the correct labelling).
Let $(a,b),(a',b')\in E(\phi(R))$ and write $(i,j):=\phi^{-1}(a,b)$ and $(i',j')=\phi^{-1}(a',b')$.
Then $(i,j)(i',j')\in E(R)$, so $G[V_{i,j},V_{i',j'}]$ is $(\eps,\delta)$-regular by~$(\mathscr{C}2)$ for $\mathcal{V}$.
Then~(\ref{Xtriangle}) implies that we can apply Proposition~\ref{newsuperslice} with $\alpha := 3r\eps$ and $\eps ' := \eps^{1/3} \geq \eps + 6\sqrt{\alpha}$ to see that $G[X_{a,b},X_{a',b'}]$ is $(\eps^{1/3},\delta/2)$-regular.
So~$(\mathscr{C}2)$ holds.

For~($\mathscr{C}3)$, fix $a \in [2\ell]$ and let $b,b' \in [r]$ be distinct.
Let $(i,j):=\phi^{-1}(a,b)$.
Definition~\ref{bijdef} implies that there exists $j'$ such that $(j,j')\in[[2r]]^2$ and $\phi^{-1}(a,b')=(i,j')$.
Let $x \in X_{a,b}\setminus V_{\phi^{-1}(a,b)}  = W_{i,j}\setminus Y_{i,j}$.
Then~$(\mathscr{W}3)$ implies that $x \rightarrow Y_{i,j}$ is valid.
Since $Y_{i,j}\subseteq V_{i,j}$, this means that $d_G(x,V_{i,j^*}) \geq (\delta-2\eps)m$ for all $j^*$ such that $(j,j^*)\in[[2r]]^2$.
So $d_G(x,V_{\phi^{-1}(a,b')}) \geq (\delta-2\eps)m$, and hence  (\ref{Xtriangle}) implies $d_G(x,X_{a,b'}) \geq (\delta-2\eps)m - 2r\eps m \geq \delta|X_{a,b'}|/2$.
Similarly, every $y \in X_{a,b'} \setminus V_{\phi^{-1}(a,b')}$ satisfies $d_G(y,X_{a,b}) \geq \delta|X_{a,b}|/2$.
Moreover,~($\mathscr{C}3)$ for $\mathcal V$  and (\ref{Xtriangle}) implies that $d_G(x,X_{a,b'})  \geq \delta|X_{a,b'}|/2$ for every $x \in V_{\phi^{-1}(a,b)}$ and  $d_G(y,X_{a,b}) \geq \delta|X_{a,b}|/2$ for every $y \in V_{\phi^{-1}(a,b')}$.
So Proposition~\ref{newsuperslice} applied with $\alpha := 3r\eps$ and $\eps' := \eps^{1/3}$ implies that $G[X_{a,b},X_{a,b'}]$ is $(\eps^{1/3},\delta/2)$-superregular.
So~$(\mathscr{C}3)$ holds.
This completes the proof of~$(\mathscr{L}2)$ and hence of the lemma.
\end{proof}

\section{The proof of Theorem~\ref{main}}\label{sec8}
First note that it suffices to prove the theorem under the additional assumption that $\eta \ll d,1/\Delta $.
Let $n_0,\beta,\rho,\eps,c,\delta,\rho ', L'>0$ satisfy
\begin{equation}\label{hierarchy}
0 < 1/n_0 \ll \beta \ll 1/L'\ll \rho \ll \eps \ll c \ll \delta \ll \rho '  \ll \eta \ll d,1/\Delta.
\end{equation}
\COMMENT{AT: delete all less than $1/3$ as we want to be dealing with $r=2$ too!}

Let $G$ be a $(\rho,d)$-dense graph on $n \geq n_0$ vertices with $\delta(G) \geq (1/2+\eta)n$.
Let $H$ be a graph on $n$ vertices with  $\Delta(H) \leq \Delta$ and bandwidth at most $\beta n$.
Write $r:=\chi (H)$; so as $\eta \ll 1/\Delta$, certainly $\eta \ll 1/r$.

Apply  the Regularity lemma (Lemma~\ref{reg}) with parameters $\eps,(4r+1)L'$ to obtain $L^* \in \mathbb{N}$.
We may assume that $\beta \ll 1/L^*$.

\COMMENT{AT: go back and check that when I talk about $r^*$, $4r$ etc earlier on, we use the correct term}
\begin{claim}\label{claim1}
There exists $L' \leq \ell \leq L^*$, a partition $\mathcal{V} = \lbrace V_0\rbrace \cup \lbrace V_{i,j} : (i,j) \in [\ell]\times[4r]\rbrace$ of $V(G)$ with $|V_{i,j}|=:m$ for all $(i,j)\in[\ell]\times[4r]$, a graph $R$ on vertex set $[\ell]\times[4r]$ and a spanning subgraph $G'$ of $G$ such that
\begin{itemize}
\item[(i)] $R$ is $(\rho ',d)$-dense;
\item[(ii)] $\delta(R) \geq (1/2+\eta/3)|R|$;
\item[(iii)] $G'$ has an $(R,\ell,4r,\mathcal{V},7\eps^{1/4},\delta/2)$-cycle structure $\mathcal{C}$ and $|V_0| \leq 2\eps^{1/2}n$;
\item[(iv)] $R[\lbrace(1,1),\ldots,(1,4r)\rbrace] \cong K_{4r}$ and $\lbrace (1,1),\ldots,(1,4r)\rbrace$ lies in a copy of $K_{324r/\eta ^2}$ in $R$.
\end{itemize}
\end{claim}

\begin{claimproof}
Apply Lemma~\ref{reg} to $G$ with parameters $\eps,\delta,(4r+1)L'$ to obtain clusters $V_1,\ldots,V_{L}$ of size $m'$, an exceptional set $V_0'$, a pure graph $G'$ and a reduced graph $R'$. So
\begin{equation}\label{L''m'}
Lm' \leq n \leq Lm' + \eps n,
\end{equation}
and $|R'|=L$ where
\begin{equation}\label{Ldef}
(4r+1)L' \leq L \leq L^*
\end{equation}
and $|V_0'| \leq \eps n$,
\begin{equation}\label{G'}
\delta(G') \geq (1/2+\eta)n - (\delta+\eps)n \geq (1/2+\eta/2)n
\end{equation}
and $G'[V_i,V_j]$ is $(\eps,\delta)$-regular whenever $ij \in E(R')$.
Lemma~\ref{inherit} implies that $R'$ is $(3\delta,d)$-dense and $\delta(R') \geq (1/2+\eta/2)L$.

Let $r^* := 324r/\eta ^2$.
Apply Theorem~\ref{rcycle} with $R',L,r^*-1,4r,3\delta,d,\eta /2$ playing the roles of $G,n,r,s,\rho,d,\eta$ to obtain 
an $(r^*-1)$-cycle $C \cong C^{r^*-1}_{4r\ell} \subseteq R'$ of order $4r \ell$ where
\begin{equation}\label{leq}
(1-\eps) L \leq L -4r \leq 4r\ell \leq L.
\end{equation}
Relabel those clusters of $R'$ corresponding to vertices of $C$
so that they are now $\lbrace V_{i,j}' : (i,j) \in [\ell]\times[4r]\rbrace$, and
\begin{equation}\label{ordering}
(1,1)(1,2)\ldots (1,4r)(2,1)\ldots (2,4r)\ldots (\ell,1)\ldots (\ell,4r) = C^{r^*-1}_{4r\ell} {\supseteq} Z^{4r}_{\ell}.
\end{equation}
\COMMENT{AT: this last inclusion doesn't follow from (1) -rather the definitions - so have removed the (1) here}
Let $R := R'[V(C)]$. So $V(R)=[\ell]\times[4r]$.
Observe that $\lbrace (1,1),\ldots,(1,4r)\rbrace$ lies in a copy of $K_{r^*}$ in $R$.
For all $i \in [\ell]$ let
\begin{equation*}\label{Ti}
T(i) := \textstyle{R\left[\bigcup_{j \in [4r]}(i,j)\right]} \stackrel{(\ref{ordering})}{\cong} K_{4r}.
\end{equation*}
Apply Lemma~\ref{superslice} with $G'[\bigcup_{j \in [4r]}V'_{i,j}],T(i),4r-1,4r,V'_{i,1},\ldots,V'_{i,4r},m',\eps,\delta$ playing the roles of $G,R, \Delta,$ $L,V_1,\ldots,V_L,m,\eps,d$ to obtain for each $j \in [4r]$ a subset $V_{i,j} \subseteq V'_{i,j}$ of size
\begin{equation}\label{m}
|V_{i,j}| = m := (1-\sqrt{\eps})m'
\end{equation}
such that for every distinct $j,j'\in[4r]$ the graph $G'[V_{i,j},V_{i,j'}]$ is $(4\sqrt{\eps},\delta/2)$-superregular.
Let $V_0 := V(G)\setminus \bigcup_{(i,j)\in[\ell]\times[4r]}V_{i,j}$ and
\begin{equation}\label{Vpart}
\mathcal{V} := \lbrace V_0 \rbrace \cup \lbrace V_{i,j} : (i,j)\in[\ell]\times[4r]\rbrace.
\end{equation}
We have
\begin{align}\label{4rell}
n &\geq 4r\ell m \stackrel{(\ref{m})}{=} 4r\ell(1-\sqrt{\eps})m' \stackrel{(\ref{leq})}{\geq} (1-\sqrt{\eps})(1-\eps)Lm' \stackrel{(\ref{L''m'})}{\geq} (1-\sqrt{\eps})(1-\eps)^2 n\\
\nonumber&\geq (1-2\eps^{1/2})n. 
\end{align}
Since we will often compare $m$ and $\beta n$ in calculations, let us note here that
\begin{equation}\label{beta}
\beta n \stackrel{(\ref{L''m'})}{\leq} \frac{\beta Lm'}{1-\eps} \stackrel{(\ref{m})}{=} \frac{\beta L m}{(1-\sqrt{\eps})(1-\eps)} \stackrel{(\ref{hierarchy}),(\ref{Ldef})}{\leq} 2\beta L^* \cdot m \leq \frac{\eps^2 m}{L^*}.
\end{equation}

We will now show that $\ell$, $R$ and $\mathcal{V}$ satisfy~Claim~\ref{claim1}(i)--(iv).
We have that
\begin{equation*}\label{eqL}
4rL' \stackrel{(\ref{hierarchy})}{\leq} (1-\eps)(4r+1)L'  \stackrel{(\ref{Ldef})}{\leq} (1-\eps)L \stackrel{(\ref{leq})}{\leq} 4r\ell \leq L \stackrel{(\ref{Ldef})}{\leq} L^*.
\end{equation*}
So $L' \leq \ell \leq L^*$, as required.
Note that (i) follows from 
Lemma~\ref{nbrhd}(i) since $\rho ' \gg \delta$.
Further, $\delta(R) \geq \delta(R')-4r \geq (1/2+\eta/3)L$ so~(ii) holds.

For~(iii), we need to show that $\mathcal{V}$ (see~(\ref{Vpart})) induces the required cycle structure $\mathcal{C}$.
That is, we need to check that $(\mathscr{C}1)$--$(\mathscr{C}3)$ hold with the desired parameters.
The sets $V_{i,j}$ are pairwise-disjoint since the same is true for $V'_{i,j}$, so by the definition of $V_0$ we have that $\mathcal{V}$ is a partition of $V(G')$. Moreover,
\begin{eqnarray*}
|V_0| = n - 4r\ell m &\stackrel{(\ref{4rell})}{\leq}& 2\eps^{1/2}n < 7\eps^{1/4}n,
\end{eqnarray*}
so~$(\mathscr{C}1)$ holds.
Certainly $V(R)=[\ell]\times[4r]$ and, by~(\ref{ordering}), $R \supseteq Z^{4r}_{\ell}$. Let $(i,j)(i',j')\in E(R)$.
Then $(i,j)(i',j')$ has a corresponding edge in $R' \supseteq R$, so $G'[V'_{i,j},V'_{i',j'}]$ is $(\eps,\delta)$-regular.
Note that $\eps + 6\sqrt{\eps^{1/2}} \leq 7\eps^{1/4}$ and $\delta-4\eps^{1/2} \geq \delta/2$.
Thus Lemma~\ref{newsuperslice} applied with $V'_{i,j},V_{i,j},V'_{i',j'},V_{i',j'},\eps^{1/2}$ playing the roles of $A,A',B,B',\alpha$ implies that $G'[V_{i,j},V_{i',j'}]$ is $(7\eps^{1/4},\delta/2)$-regular, so~$(\mathscr{C}2)$ holds.
We have already seen, for every $i \in [\ell]$ and distinct $j,j'\in[4r]$, that $G'[V_{i,j},V_{i,j'}]$ is $(4\sqrt{\eps},\delta/2)$-superregular. Thus it is $(7\eps^{1/4},\delta/2)$-superregular. So~$(\mathscr{C}3)$ holds.
Thus~(iii) holds. We saw when we defined $R$ that~(iv) holds.
This completes the proof of the claim.
\end{claimproof}

\medskip
\noindent
Recall the definition of the bijection $\phi^{4r}_{\ell} : [\ell]\times[4r]\rightarrow[2\ell]\times[2r]$ given by
$$
\phi^{4r}_{\ell}(i,j)=\left((2i-1)+\left\lfloor\frac{j}{2r}\right\rfloor,j-\left(\left\lceil \frac{j}{2r}\right\rceil-1\right) 2r\right),\quad\text{for all }(i,j)\in[\ell]\times[4r].
$$
Recall also that $\phi^{4r}_\ell(R)$ is the graph with vertex set $\phi^{4r}_\ell(V(R)) = [2\ell]\times[2r]$ and edge set $E(\phi^{4r}_\ell(R)) = \lbrace \phi^{4r}_\ell(x)\phi^{4r}_\ell(y) : xy \in E(R)\rbrace$.
For ease of notation, we will write
\begin{eqnarray}
\label{phiprop} \phi &:=& \phi^{4r}_{\ell},\quad\text{so}\quad \phi(1,b)=(1,b)\text{ for all }b \in [2r],\quad\text{and}\\
\nonumber R^* &:=& \phi(R),\quad\text{so }R^* \cong R,\quad V(R^*)=[2\ell]\times[2r],
\end{eqnarray}
and $V(G')$ has partition $\mathcal{V} = \lbrace V_0\rbrace \cup \lbrace V_{\phi^{-1}(a,b)} : (a,b)\in[2\ell]\times[2r]\rbrace$.

\begin{claim}\label{claim2}
There exists a partition $\mathcal{X} = \lbrace V_0\rbrace \cup \lbrace X_{a,b} : (a,b)\in[2\ell]\times[2r]\rbrace$ of $V(G')$ and a surjective\COMMENT{To ensure all of $V_0$ covered} mapping $\psi : V(H) \rightarrow ([2\ell]\times[2r])\cup V_0$ such that the following hold:
\begin{itemize}
\item[(i)]  $|\psi^{-1}(a,b)|=|X_{a,b}| \geq (1-\eps^{1/19})m$ for all $(a,b)\in[2\ell]\times[2r]$;
\item[(ii)] $G'$ has an $(R^*,2\ell,2r,\mathcal{X},\eps^{1/27},\delta/4)$-cycle structure $\mathcal{C}'$;
\item[(iii)] $I := \psi^{-1}(V_0) $ is an independent set in $H$ of size $|V_0|$ and for all $w \in W := \bigcup_{x \in I}N_H(x)$, there is a unique $u \in I$ such that $uw \in E(H)$, and $d_{G'}(\psi(u),X_{\psi(w)}) \geq cm/2$; 
\item[(iv)] $\psi|_{V(H)\setminus I} : V(H\setminus I) \rightarrow V(R^*)$ is a graph homomorphism;
\item[(v)] there exists $X' \subseteq V(H)\setminus I$ with $W \subseteq X'$ and $|\psi^{-1}(a,b)\cap X'| \leq \eps^{1/10}m$ for all $(a,b)\in[2\ell]\times[2r]$ such that, whenever $uv \in E(H)$ and $u,v \notin X' \cup I$, writing $\psi(u)=:(a,b)$ and $\psi(v)=:(a',b')$, we have $a=a'$ and $b \neq b'$. Moreover, writing
$$
N := \textstyle{\left(\bigcup_{x \in X'}N_H(x)\right)\setminus (X' \cup I)},
$$
we have $|N| \leq \eps m$.
\end{itemize}
\end{claim}

\begin{claimproof}
For all $v \in V(G')$, write
\begin{equation}\label{Nalpha}
N_{R^*}^c(v) := \lbrace (a,b)\in[2\ell]\times[2r] : d_{G'}(v,V_{\phi^{-1}(a,b)})\geq c m\rbrace
\end{equation}
and $d_{R^*}^c(v) := |N_{R^*}^c(v)|$. Then
$$
(1/2+\eta) n - (\delta+\eps)n \stackrel{(\ref{G'})}{\leq} d_{G'}(v) \leq d_{R^*}^c(v)m + (4r\ell-d_{R^*}^c(v))c m + |V_0|.
$$
Claim~\ref{claim1}(iii)  implies that
\begin{equation*}\label{meq2}
4r\ell m \leq n \leq 4r\ell m + |V_0| \leq 4r\ell m+ 2\eps^{1/2} n.
\end{equation*}
Thus
\begin{equation}\label{dalpha}
d_{R^*}^c(v) \geq \frac{(1/2+\eta-\delta-\eps) n - 4r\ell cm - |V_0|}{(1-c)m} \stackrel{(\ref{hierarchy})}{\geq} \frac{1/2+ \eta/2}{1-c}\cdot 4r\ell \geq \frac{|R^*|}{2}.
\end{equation}

We would like to apply Lemma~\ref{lemmaforH} (Special Lemma for $H$) to obtain an integer $s$, with $G',R^*,4r\ell,$ $2r,\eta/3,2\eps^{1/2},\rho ',d,n,m,N_{R^*}^c(v),(1,i)$ playing the roles of $G,R,L,r,\eta,\eps,\rho,d,n,m,N_v,b_i$.
For this, we need to check that~$(\mathscr{G}1)$--$(\mathscr{G}4)$ hold.
For~$(\mathscr{G}1)$, we know that $G'$ has vertex partition $\lbrace V_0 \rbrace \cup \lbrace V_{i,j} : (i,j)\in[\ell]\times[4r]\rbrace = \lbrace V_0\rbrace\cup \lbrace V_{\phi^{-1}(a,b)} : (a,b)\in[2\ell]\times[2r]\rbrace$, and $|V_0| \leq \eps^{1/2}n$ and $|V_p|=m$ for all $p \in V(R^*)$.
That~$(\mathscr{G}2)$ holds follows from~(\ref{dalpha}).
Property~$(\mathscr{G}3)$ follows from~Claim~\ref{claim1}(i) and~(ii) and the fact that $R^* \cong R$.
Finally,~$(\mathscr{G}4)$ follows from~(iv), noting that $324r/\eta ^2= 18 \cdot (2r) \cdot 1/(\eta/3)^2$, and the fact that $\phi(1,b)=(1,b)$ for all $b \in [2r]$ from~(\ref{phiprop}).
Therefore we can apply Lemma~\ref{lemmaforH} with the above parameters to obtain an integer
\begin{equation}\label{seq}
s \leq (2\eps^{1/2})^{1/4}n \leq \eps^{1/9}n.
\end{equation}

Let $\chi : V(H)\rightarrow[r]$ be a proper colouring of $H$, let $x_1,\ldots,x_n$ be an ordering of $V(H)$ with bandwidth at most $\beta n$, and
let
\begin{align}
\label{XYZdef} X &:= \lbrace x_1,\ldots,x_{s}\rbrace,\quad Y := \lbrace x_{s+1},\ldots,x_{s+\beta n}\rbrace \subseteq Z := \lbrace x_{s+ 1},\ldots,x_n\rbrace,\\
\label{H'def} H' &:= H[X \cup Y]\quad\text{and}\quad H'' := H[Z].
\end{align}
Apply Lemma~\ref{lemmaforH} (Special Lemma for $H$) with the above parameters and with $s,\beta,\Delta,H',X,Y,\chi$ playing the roles of $s,\beta,\Delta,H,X,Y,\chi$ to obtain a mapping
\begin{equation*}\label{ffinaldef}
f:X \cup Y\rightarrow ([2\ell]\times[2r]) \cup V_0
\end{equation*}
with the following properties:
\begin{itemize}
\item[$(\mathscr{D}1)$]  
setting $I := f^{-1}(V_0)$, we have that $I$ is a subset of $X$ which is $2$-independent in $H'$, and each  vertex in $V_0$ is mapped onto from a unique vertex in $H$ (so $|I|=|V_0|$);
\item[$(\mathscr{D}2)$] for all $v \in V_0$, setting $W_v := N_H(f^{-1}(v))$, we have $W_v \subseteq X$ and $f(W_v) \subseteq N_{R^*}^c(v)$;
\item[$(\mathscr{D}3)$] $|f^{-1}(a,b)| \leq \eps^{1/9}m$ for every $(a,b) \in [2\ell]\times[2r]$;
\item[$(\mathscr{D}4)$] for every edge $uv\in E(H)$ such that $f(u),f(v) \notin V_0$, we have $f(u)f(v) \in E(R^*)$;
\item[$(\mathscr{D}5)$] for all $y \in Y$ we have $f(y)=(1,\chi(y))$.
\end{itemize}

Let
\begin{equation}\label{taudef}
\tau_{a,b} := |(f|_{X\setminus I})^{-1}(a,b)| = |(f|_X)^{-1}(a,b)| \quad\text{for all }(a,b)\in[2\ell]\times[2r].
\end{equation}
Then $0 \leq \tau_{a,b} \leq \eps^{1/9}m$ for all $(a,b)\in[2\ell]\times[2r]$ by~$(\mathscr{D}3)$. 

\smallskip

Apply Lemma~\ref{lemmaforG} (the Lemma for $G$) with $n-|V_0|,\ell,m,2r,11\beta,\eps^{1/9},\delta/2,G'\setminus V_0,R,\mathcal{V}\setminus \lbrace V_0\rbrace$ playing the roles of $n,\ell,m,r,\xi,\eps,\delta,G,R,\mathcal{V}$ to obtain
positive integers $\lbrace m_{a,b} : (a,b)\in[2\ell]\times[2r]\rbrace$ such that
\begin{itemize}
\item[$(\mathscr{L}1)$] $\sum_{(a,b)\in[2\ell]\times[2r]}(m_{a,b}+\tau_{a,b})=n-|V_0|$ and $m_{a,b} \geq (1-\eps^{1/18})m$ and $|m_{a,b}-m_{a,b'}| \leq 1$ for all $a \in [2\ell]$ and $b,b' \in [2r]$;
\item[$(\mathscr{L}2)$] given any $\lbrace n_{a,b} : (a,b) \in [2\ell]\times[2r]\rbrace$ with $\sum_{(a,b)\in[2\ell]\times[2r]}(n_{a,b}+\tau_{a,b})=n-|V_0|$ and $|m_{a,b}- n_{a,b}| \leq 11\beta (n-|V_0|)$, there is a partition $\mathcal{X} = \lbrace V_0 \rbrace \cup \lbrace X_{a,b} : (a,b)\in[2\ell]\times[2r] \rbrace$ of $V(G')$ with $|X_{a,b}| = n_{a,b}+\tau_{a,b}$ and $|X_{a,b}\bigtriangleup V_{\phi^{-1}(a,b)}| \leq \eps^{1/18}m$ for all $(a,b)\in[2\ell]\times[2r]$ such that
$G'$ has an $(R^*,2\ell,2r,\mathcal{X},\eps^{1/27},\delta/4)$-cycle structure.
\end{itemize}
Note that Lemma~\ref{lemmaforG} yields a partition of $G'\setminus V_0$ into clusters, and the partition of $V(G')$ specified in~$(\mathscr{L}2)$ is simply this partition together with $V_0$.

\smallskip

The next step is to apply Lemma~\ref{lemmaforH1} (Basic Lemma for $H$) to $H'' = H[Z]$ (which overlaps with $H'$ in $Y$). Note that the number of vertices in $H''$ is $n-s \geq (1-\eps^{1/9})n$. Further,
\begin{equation}\label{mab}
\sum_{(a,b)\in[2\ell]\times[2r]}m_{a,b} \stackrel{(\mathscr{L}1)}{=} n - |V_0| - \sum_{(a,b)\in[2\ell]\times[2r]}\tau_{a,b} \stackrel{(\ref{taudef})}{=} n- |V_0| - |X\setminus I| \stackrel{(\mathscr{D}1)}{=} n-|X| \stackrel{(\ref{XYZdef})}{=}|Z|
\end{equation}
and $m_{a,b} \geq (1-\eps^{1/18})m \geq 10\beta(n-s)$ by~(\ref{hierarchy}).
Thus we can apply Lemma~\ref{lemmaforH1} with $n-s,r,2\ell,\Delta,\beta,H'',(x_{s+1},\ldots,x_n),\chi,\lbrace m_{a,b} : (a,b)\in[2\ell]\times[2r]\rbrace$ playing the roles of
$
n,r,\ell,\Delta,\beta,H,$ $(x_1,\ldots,x_n),\chi,\lbrace m_{a,b}:(a,b)\in[2\ell]\times[2r]\rbrace
$
to obtain a mapping
\begin{equation}\label{kdef}
k : Z \rightarrow [2\ell]\times[2r]
\end{equation}
and $B \subseteq Z$ with the following properties: 
\begin{itemize}
\item[$(\mathscr{B}1)$] $B \subseteq Z\setminus Y$ and $|B| \leq 2\ell\beta n$;
\item[$(\mathscr{B}2)$] $\left| |k^{-1}(a,b)| - m_{a,b}\right| \leq 10\beta n$ for every $(a,b) \in [2\ell]\times[2r]$;
\item[$(\mathscr{B}3)$] for every edge $uv\in E(H'')$, writing $k(u)=:(a,b)$ and $k(v)=:(a',b')$, we have $|a-a'|\leq 1$ and $b \neq b'$. If additionally $u,v \notin B$, then $a=a'$;
\item[$(\mathscr{B}4)$] for all $y \in Y$ we have $k(y)=(1,\chi(y))$.
\end{itemize}
Let
\begin{equation}\label{nabdef}
n_{a,b} := |k^{-1}(a,b)|\quad\text{for all }(a,b)\in[2\ell]\times[2r].
\end{equation}
Then
$$
\sum_{(a,b)\in[2\ell]\times[2r]}n_{a,b} \stackrel{(\ref{kdef}),(\ref{nabdef})}{=} |Z| \stackrel{(\ref{mab})}{=} \sum_{(a,b)\in[2\ell]\times[2r]}m_{a,b} \stackrel{(\mathscr{L}2)}{=} n - |V_0| - \sum_{(a,b)\in[2\ell]\times[2r]}\tau_{a,b}
$$
and for all $(a,b)\in[2\ell]\times[2r]$,
$$
|n_{a,b}-m_{a,b}|\stackrel{(\ref{nabdef})}{=}||k^{-1}(a,b)|-m_{a,b}| \stackrel{(\mathscr{B}2)}{\leq} 10\beta n \leq 11\beta(n-|V_0|).
$$
Thus~$(\mathscr{L}2)$ implies that there is a partition $\mathcal{X} = \lbrace V_0 \rbrace \cup \lbrace X_{a,b} : (a,b)\in[2\ell]\times[2r] \rbrace$ of $V(G')$ with, for all $(a,b)\in[2\ell]\times[2r]$,
\begin{equation}\label{Xprops}
|X_{a,b}| = |k^{-1}(a,b)|+|(f|_{X\setminus I})^{-1}(a,b)|\quad\text{and}\quad|X_{a,b}\bigtriangleup V_{\phi^{-1}(a,b)}| \leq \eps^{1/18}m
\end{equation}
such that
$G'$ has an $(R^*,2\ell,2r,\mathcal{X},\eps^{1/27},\delta/4)$-cycle structure.

Define a mapping $\psi : V(H) \rightarrow ([2\ell]\times[2r]) \cup V_0$ by
setting
\begin{equation}\label{psidef99}
\psi(x) = \begin{cases} f(x) &\mbox{if } x\in X \\
k(x) & \mbox{if } x \in Z.\end{cases} 
\end{equation}
Finally, let $X' := (X\setminus I) \cup B$.

We need to check that $\mathcal{X}$, $\psi$ and $X'$ satisfy Claim~\ref{claim2}(i)--(v).
For~(i), we have
\begin{eqnarray*}
|\psi^{-1}(a,b)| &\stackrel{(\ref{psidef99})}{=}& |k^{-1}(a,b)|+|(f|_{X\setminus I})^{-1}(a,b)| \stackrel{(\ref{Xprops})}{=} |X_{a,b}|  \stackrel{(\mathscr{B}2)}{\geq} m_{a,b}-10\beta n\\
&\stackrel{(\mathscr{L}1)}{\geq}& (1-\eps^{1/18})m-10\beta n \stackrel{(\ref{beta})}{\geq} (1-\eps^{1/19})m.
\end{eqnarray*}
Further, we have already seen that~(ii) holds.

Note that $I =f^{-1}(V_0)=\psi^{-1}(V_0)$ has size $|V_0|$ and is a $2$-independent subset of $X$ in $H'$ by~$(\mathscr{D}1)$. Let $w \in W := \bigcup_{x \in I}N_H(x)$.
Since $I$ is $2$-independent, there is a unique $u \in I \subseteq X$ such that $uw \in E(H)$. So $w \in W_{f(u)}=W_{\psi(u)} \subseteq X$ in the notation of~$(\mathscr{D}2)$.
So $\psi(w) = f(w) \in N_{R^*}^c(f(u))=N_{R^*}^c(\psi(u))$.
Thus
$$
d_{G'}(\psi(u),X_{\psi(w)}) \stackrel{(\ref{Xprops})}{\geq} d_{G'}(\psi(u),V_{\phi^{-1}(\psi(w))}) - \eps^{1/18}m \stackrel{(\ref{Nalpha})}{\geq} cm/2,
$$
so~(iii) holds.

For~(iv), note that $k(y)=(1,\chi(y))  = f(y)=\psi(y)$ for all $y \in Y$ by~$(\mathscr{D}5)$ and~$(\mathscr{B}4)$.
Observe that $\psi' := \psi|_{V(H)\setminus I}$ is a map into $V(R^*)=[2\ell]\times[2r]$. Let $xy \in E(H)$ where $x,y \notin I$.
Suppose first that $x,y \in X \cup Y$. Then $\psi(x)=f(x)$ and $\psi(y)=f(y)$.
Then $f(x),f(y) \notin V_0$, so~$(\mathscr{D}4)$ implies that $f(x)f(y) \in E(R^*)$.
Suppose now that $x,y \in Z$.
Write $\psi(x)=k(x)=(a,b)$ and $\psi(y)=k(y)=(a',b')$ where $(a,b),(a',b')\in[2\ell]\times[2r]$.
Then~$(\mathscr{B}3)$ implies that $|a-a'| \leq 1$ and $b \neq b'$. Thus $\psi(x)\psi(y) \in E(Z^{2r}_{2\ell}) \subseteq E(R^*)$, as required.
The only other possibility is that one of $x,y$ is in $X$ and the other is in $Z\setminus Y$.
But then the distance between them in the bandwidth ordering of $H$ is more than $|Y|=\beta n$, a contradiction to $xy \in E(H)$.
Thus $\psi' : V(H\setminus I) \rightarrow V(R^*)$ is a graph homomorphism.
So~(iv) holds.

For~(v), 
note that $B \subseteq Z$ so $X' \cap I = \emptyset$, and $W = \bigcup_{v \in V_0}W_v \subseteq X$, and $W \cap I = \emptyset$ since $I$ is $2$-independent in $H'$. So $W \subseteq X'$.
Now
let $(a,b)\in[2\ell]\times[2r]$. We have
$$
|\psi^{-1}(a,b) \cap X'| \leq |\psi^{-1}(a,b) \cap X|+|B| \leq |f^{-1}(a,b)| + |B| \stackrel{(\mathscr{D}3),(\mathscr{B}1)}{\leq} \eps^{1/9}m + 2\ell\beta n \stackrel{(\ref{beta})}{\leq} \eps^{1/10}m.
$$
Now let $uv \in E(H)$ where $u,v \notin X' \cup I $. So $u,v \in Z \setminus B$. Write $\psi(u)=(a,b)$ and $\psi(v)=(a',b')$.
Then $(a,b)=k(u)$ and $(a',b')=k(v)$, and~$(\mathscr{B}3)$ implies that $a=a'$ and $b \neq b'$, as required.

Finally, define $N$ as in~(v).
If $y \in N$, then either $y \in \bigcup_{x \in X}N_H(x)\setminus X \subseteq Y$; or $y \in \bigcup_{x \in B}N_H(x)$ (or both).
So~(\ref{XYZdef}) and the fact that $\Delta(H) \leq \Delta$ implies that
\begin{equation}\label{N}
|N| \leq |Y|+\Delta|B| \leq (2\Delta\ell+1)\beta n \stackrel{(\ref{hierarchy}),(\ref{beta})}{\leq} \eps m.
\end{equation}
This completes the proof of~(v) and hence of the claim.
\end{claimproof}

\medskip
\noindent
In the final part of the proof, we will use the cycle structure $\mathcal{C}'$, mapping $\psi$ and special set $X'$ obtained in Claim~\ref{claim2} to find an embedding $g$ of $H$ into $G'\subseteq G$. We will do this in three stages: (1) first define an embedding $g_1$ of $I$ into $V_0$, according to $\psi$; (2) find an embedding $g_2$ of $X'$ using $\psi$ as a framework, such that there are large candidate sets for the neighbouring vertices $N$ of $X'$; (3) find an embedding $g_3$ of the remainder of $H$ using the Blow-up lemma, using the candidate sets obtained in (2) to ensure that $g_2$ is compatible with $g_3$.
Then set $g$ to be the union of $g_1,g_2,g_3$.

Stage (1) is easy; we simply define
$$
g_1 : I \rightarrow V_0\quad\text{where}\quad g_1(x) := \psi(x)\quad\text{for all }x \in I.
$$
Since by Claim~\ref{claim2}(iii), $I$ is an independent set in $H$ of size $V_0$, we trivially have that $g_1$ is an embedding of $H[I]$ into $V(G')$.

For Stage~(2), we will apply Lemma~\ref{target} (Embedding lemma with target sets) to embed vertices in $X'$.
Indeed, let $\psi^* := \psi|_{X' \cup N}$.
Given $w \in W$, let $u$ be the unique element of $I$ such that $uw \in E(H)$, as guaranteed by Claim~\ref{claim2}(iii). Let
\begin{equation}\label{Sw99}
S_w := N_{G'}(\psi(u),X_{\psi(w)}).
\end{equation}
We will apply Lemma~\ref{target} with
$
G'\setminus V_0,R^*,H[X' \cup N],n-|V_0|,4r\ell,\eps^{1/27},c/2,\delta/4,\Delta,$ $ \lbrace X_{a,b}\rbrace,(1-\eps ^{19})m,\psi^*,X',$ $N,W,S_w
$
playing the roles of
$G,R,H,
n,L,\eps,c,\delta,\Delta,\lbrace V_a : a \in V(R)\rbrace,m,\phi,X,Y,W,S_w
$.
To see why this is possible, note that, by Claim~\ref{claim1}(iii),
$G'\setminus V_0$ has $n-|V_0| \geq (1-2\eps^{1/2})n$ vertices and Claim~\ref{claim2}(ii) (specifically~$(\mathscr{C}2)$) implies that it has an $(\eps^{1/27},\delta/4)$-regular partition $\lbrace X_{a,b} : (a,b)\in V(R^*)\rbrace$.
Clearly, as a restriction of $\psi$, the function $\psi^*$ is a suitable graph homomorphism, and by Claim~\ref{claim2}(v) and~(\ref{N}), we have
\begin{equation}\label{psistar}
|(\psi^*)^{-1}(a,b)| \leq |\psi^{-1}(a,b) \cap X'| + |N| \leq \eps^{1/10}m + \eps m \leq \eps^{1/12}m.
\end{equation}
Finally, $W \subseteq X'$ by Claim~\ref{claim2}(v), and $|S_w| \geq cm/2$ by Claim~\ref{claim2}(iii).
So the above are suitable parameters for the application of Lemma~\ref{target}.

Thus there is a mapping
$$
g_2 : X' \rightarrow V(G')\setminus V_0
$$
which is an embedding of $H[X']$ into $G'$ such that
\begin{itemize}
\item[$(\mathscr{T}1)$] $g_2(x) \in X_{\psi^*(x)}$ for all $x \in X'$;
\item[$(\mathscr{T}2)$] $g_2(w) \in S_w$ for all $w \in W$;
\item[$(\mathscr{T}3)$] for all $y \in N$ there exists $C_y \subseteq X_{\psi^*(y)}\setminus g_2(X')$ such that $C_y \subseteq N_{G'}(g_2(x))$ for all $x \in N_H(y) \cap (X')$, and $|C_y| \geq cm/2$.
\end{itemize}

For Stage~(3), we will do the following for each $a \in [2\ell]$.
Let $U_{a,b} := X_{a,b} \setminus g_2(X')$ for all $b \in [2r]$.
We want to show that $U_{a,b}$ has exactly the right size to embed the remaining vertices of $H$ whose image under $\psi$ is $(a,b)$.
Indeed, let $\psi' := \psi|_{H\setminus (X' \cup I)}$.
Then Claim~\ref{claim2}(i) implies that
$$
|U_{a,b}| = |X_{a,b}| - |g_2(X') \cap X_{a,b}| \stackrel{(\mathscr{T}1)}{=} |\psi^{-1}(a,b)| - |(\psi^*)^{-1}(a,b) \cap X'| = |(\psi')^{-1}(a,b)|
$$
where we used the fact that $\psi^{-1}(a,b) \cap I=\emptyset$.
This together with~(\ref{psistar}) implies that
$|U_{a,b}\bigtriangleup X_{a,b}| = |(\psi^*)^{-1}(a,b) \cap X'| \leq \eps^{1/10}m \leq 2\eps^{1/10}|U_{a,b}|$.
Let $b,b' \in [2r]$ be distinct.
So $|U_{a,b}| \geq (1-\eps^{1/20})m$ by Claim~\ref{claim2}(i).
Recall from Claim~\ref{claim2}(ii) (specifically~$(\mathscr{C}3)$) that $G'[X_{a,b},X_{a,b'}]$ is $(\eps^{1/27},\delta/4)$-superregular.
So given any $x \in U_{a,b}$, Claim~\ref{claim2}(i) implies that 
$$
d_{G'}(x,U_{a,b'}) \geq \delta|X_{a,b}|/4-\eps^{1/10}m \geq (\delta/4-\delta\eps^{1/19}-\eps^{1/10})m \geq \delta|U_{a,b}|/5.
$$
Thus Proposition~\ref{newsuperslice} with $G',X_{a,b},U_{a,b},X_{a,b'},U_{a,b'},\eps^{1/27},\delta/4,\eps^{1/10}$ playing the roles of $G,A,A',B,B',$ $\eps,\delta,\alpha$ implies that
 $G'[U_{a,b},U_{a,b'}]$ is $(2\eps^{1/27},\delta/5)$-superregular for all distinct $b,b'\in[2r]$.
The set $N$ has size at most $\eps m \leq 2\eps|U_{a,b}|$ for any $b \in [2r]$, and for each $y \in N \cap (\psi')^{-1}(a,b)$,~$(\mathscr{T}3)$ guarantees a corresponding set $C_y \subseteq X_{\psi^*(y)}\setminus g_2(X') = U_{\psi'(y)} = U_{a,b}$ which has size at least $cm/2 \geq c|U_{a,b}|/3$.
Let $H_a$ denote the subgraph of $H$ induced by the set of all $x \in V(H) \setminus (X' \cup I)$ such that $\psi '(x)=(a,b)$ for some $b \in [2r]$. \COMMENT{AT: note all the changes here}
Now apply, for each $a \in [2\ell]$, Lemma~\ref{blowlem} (Blow-up Lemma) with 
$G'[\bigcup_{b \in [2r]}U_{a,b}]$ and $H_a$ playing the roles of $G$ and $H$ and
$
2\eps^{1/27},2\eps,\delta/5,c/3,\Delta,2r,\lbrace U_{a,b}: b \in [2r]\rbrace,\psi',C_y
$
playing the roles of
$
\eps,\alpha,\delta,c,\Delta,k,\lbrace V_a:a \in [k]\rbrace,\phi,S_y.
$
Altogether this yields a mapping
$$
g_3 : V(H)\setminus (X' \cup I) \rightarrow V(G')\setminus (V_0 \cup g_2(X'))
$$
which is an embedding of $H\setminus(X' \cup I)$ into $V(G')$ such that every $y \in N$ is mapped to a vertex in $C_y$.

We claim that the mapping $g$ given by
\begin{equation}\label{psidef}
g(x) = \begin{cases} g_1(x) &\mbox{if } x\in I \\
g_2(x) &\mbox{if } x\in X' \\
g_3(x) & \mbox{otherwise}.\end{cases} 
\end{equation}
is an embedding of $H$ into $G'$ (and hence into $G$).

Firstly, $g$ is an injective map from $V(H)$ to $V(G')$ by the definitions of $g_1,g_2,g_3$.
So we just need to check that it is a graph homomorphism.
Also by their definitions, each of $g_1,g_2,g_3$ is an embedding of $H$ induced on their respective domains into $G'$.
So it suffices to check that whenever $xy \in E(H)$ and $x,y$ are not both in $I$ or in $X'$ or in $V(H)\setminus (X' \cup I)$, that $g(x)g(y)\in E(G')$.

Suppose first that $x \in I$ and $y \in V(H)\setminus I$. Then $g(x)=g_1(x)=\psi(x)$ and $y \in W \subseteq X'$ (here we used Claim~\ref{claim2}(v)). So $g(y)=g_2(y)$.
Claim~\ref{claim2}(iii) implies that $x$ is the only vertex in $I$ which is a neighbour of $y$.
Then
$$
g(y)\stackrel{(\ref{psidef})}{=}g_2(y) \stackrel{(\mathscr{T}2)}{\in} S_y \stackrel{(\ref{Sw99})}{=} N_{G'}(\psi(x),X_{\psi(y)}) = N_{G'}(g(x),X_{\psi(y)}).
$$
So $g(x)g(y)\in E(G')$, as required.

Therefore we may assume that $x \in X'$ and $y \in V(H)\setminus (X' \cup I)$.
Then $g(x)=g_2(x)$, $y \in N$ and $g(y)=g_3(y) \in C_y$, where $C_y$ was defined in~$(\mathscr{T}3)$, which guarantees that $C_y \subseteq N_{G'}(g_2(x)) = N_{G'}(g(x))$.
So $g(x)g(y)\in E(G')$, as required.
This completes the proof of Theorem~\ref{main}.


\section{Concluding remarks}\label{sec9}
In this paper we  prove a version of the Bandwidth theorem for locally dense graphs. As mentioned in the introduction, it is also of interest to seek minimum degree conditions that force a given spanning structure in a graph with
sublinear independence number. In particular, it would be very interesting to obtain an analogue of the Bandwidth theorem in this setting.

In a step in this direction,
Balogh, Molla and  Sharifzadeh~\cite{bms} proved the following result on triangle factors.
\begin{theorem}[Balogh, Molla and  Sharifzadeh~\cite{bms}]\label{bmsthm}
For every $\eps >0$, there exists $\gamma > 0$ and $n_0 \in \mathbb N$ such that the following holds.
For every $n$-vertex graph $G$ with $n \geq n_0$ divisible by $3$, if $\delta (G) \geq (1/2 + \eps)n$ and $G$ has independence number $\alpha(G) \leq \gamma n$, then $G$ has
a $K_3$-factor.
\end{theorem}
Perhaps the next natural step is to ascertain whether the conclusion of Theorem~\ref{bmsthm} can be strengthened to ensure the square of a Hamilton cycle.
\begin{conjecture}
For every $\eps >0$, there exists $\gamma > 0$ and $n_0 \in \mathbb N$ such that the following holds.
For every $n$-vertex graph $G$ with $n \geq n_0$, if $\delta (G) \geq (1/2 + \eps)n$ and  $\alpha(G) \leq \gamma n$, then $G$ 
contains the square of a Hamilton cycle.
\end{conjecture}

It is also natural to seek a version of Theorem~\ref{main} where now one replaces the condition of locally dense with a more restrictive \emph{uniformly dense} condition:
given $\rho,d >0$, we say that an $n$-vertex graph $G$ is \emph{$(\rho,d)$-uniformly-dense} 
if every $X ,Y\subseteq V(G)$  satisfies $e_G(X,Y) \geq d|X||Y|-\rho n^2$.
If one restricts to uniformly dense graphs, then one can substantially reduce the minimum degree condition in Theorem~\ref{main}, as well as remove the bandwidth condition on $H$.
\begin{theorem}\label{udthm}
For all $\Delta \in \mathbb{N}$ and $d,\eta>0$, there exist constants $\rho,n_0>0$ such that for every $n \geq n_0$, the following holds.
Let $H$ be an $n$-vertex graph  with $\Delta(H) \leq \Delta$.
Then any $(\rho,d)$-uniformly-dense graph $G$ on $n$ vertices with $\delta(G)\geq \eta n$ contains a copy of $H$.
\end{theorem}
Theorem~\ref{udthm} can be proven by a  simple application of the Blow-up lemma; a more general `rainbow' version of Theorem~\ref{udthm} is given in~\cite[Corollary 1.3]{gj}.
\section*{Acknowledgements}
We would like to thank Maryam Sharifzadeh for many helpful conversations at the start of the project. The second author is grateful to Stefan Glock and Felix Joos for a conversation on~\cite{gj} that brought to light the
version of the Bandwidth theorem for uniformly dense graphs.

The authors are also grateful to the referees for their helpful and careful reviews.

\medskip

{\footnotesize \obeylines \parindent=0pt

\begin{tabular}{lll}
  Katherine Staden &  Andrew Treglown     \\
  Mathematical Institute	&  School of Mathematics		 \\
  University of Oxford&University of Birmingham\\
  Andrew Wiles Building       &   Edgbaston           			\\
  Radcliffe Observatory Quarter&Birmingham			  			\\
  Woodstock Road&		B15 2TT			     \\
 Oxford&UK\\
OX2 6GG&\\
 UK&
\end{tabular}
}

\begin{flushleft}
{\it{E-mail addresses}:
\tt{staden@maths.ox.ac.uk},
\tt{a.c.treglown@bham.ac.uk}.}
\end{flushleft}

\end{document}